\documentclass{llncs}
\usepackage[utf8]{inputenc}
\usepackage{amsmath,comment,tikz,slashbox,boxedminipage,xypic,amssymb,fullpage,booktabs,todonotes}
\usepackage{arydshln}

\newcommand{\NP}{{\sf NP}}

\newtheorem{open}{Open Problem}

\begin{document}
\title{Complexity~Framework~For~Forbidden~Subgraphs I:\\ The Framework}
\author{Matthew Johnson\inst{1} \and Barnaby Martin\inst{1} \and Jelle J. Oostveen\inst{2} \and Sukanya Pandey\inst{2} \and\\ Dani\"el Paulusma\inst{1}  \and Siani Smith\inst{1} \and Erik Jan van Leeuwen\inst{2}}

\pagestyle{plain}

\authorrunning{M. Johnson et al.}

\institute{Department of Computer Science,
 Durham University, Durham, UK,
\email{\{matthew.johnson2,barnaby.d.martin,daniel.paulusma,siani.smith\}@durham.ac.uk}
\and
Department of Information and Computing Sciences, 
Utrecht University,\\
The Netherlands,
\email{\{j.j.oostveen,s.pandey1,e.j.vanleeuwen\}@uu.nl}
}

\maketitle

\begin{abstract}
For any particular class of graphs, algorithms for computational problems restricted to the class often rely on structural properties that depend on the specific problem at hand. This begs the question if a large set of such results can be explained by some common problem conditions. We propose such conditions for $\mathcal{H}$-subgraph-free graphs. For a  set of graphs $\mathcal{H}$, a graph $G$ is $\mathcal{H}$-subgraph-free if $G$ does not contain any of graph from ${\cal H}$ as a subgraph.  Our conditions are easy to state. A graph problem must be efficiently solvable on graphs of bounded treewidth, computationally hard on subcubic graphs, and computational hardness must be preserved under edge subdivision of subcubic graphs.  Our meta-classification says that if a graph problem satisfies all three conditions, then for every finite set ${\cal H}$, it is ``efficiently solvable'' on ${\cal H}$-subgraph-free graphs if $\mathcal{H}$ contains a disjoint union of one or more paths and subdivided claws, and is ``computationally hard'' otherwise.  We illustrate the broad applicability of our meta-classification by obtaining a dichotomy between polynomial-time solvability and \NP-completeness for many well-known partitioning, covering and packing problems, network design problems and width parameter problems. For other problems, we obtain a dichotomy between almost-linear-time solvability and having no subquadratic-time algorithm (conditioned on some hardness hypotheses).  The proposed framework thus gives a simple pathway to determine the complexity of graph problems on $\mathcal{H}$-subgraph-free graphs. This is confirmed even more by the fact that along the way, we uncover and resolve several open questions from the literature.

\keywords{forbidden subgraph; complexity dichotomy; treewidth}
\end{abstract}

\section{Introduction}\label{s-intro}

Algorithmic meta-theorems are general algorithmic results applying to a whole range of problems, rather than just a single problem alone~\cite{Kr11}.
An \emph{algorithmic meta-theorem} is a statement saying that all problems sharing some property or properties $P$, restricted to a class of inputs $I$, can be solved efficiently by a certain form of algorithm. 
Probably the most famous algorithmic meta-theorem is that of Courcelle~\cite{Co90}, which proves that every graph property expressible in monadic second-order logic is decidable in linear time if restricted to graphs of bounded {\it treewidth} (see Section~\ref{s-dicho} for a definition of treewidth). 
Another example is that of Seese~\cite{Se96}, which proves that every graph property expressible in first-order logic is decidable in linear time when restricted to graphs of bounded degree. A third example comes from Dawar et al.~\cite{DGKS06}, who proved that every first-order definable optimisation problem admits a polynomial-time approximation scheme on any class of graphs excluding at least one minor. 
There is a wealth of further algorithmic meta-theorems (see, for example, \cite{BFLPST16,DFHT05,FLS18}), many of which combine structural graph theory (e.g.\ from graph minors) with logic formulations or other broad problem properties (such as bidimensionality).

An extension of an algorithmic meta-theorem can produce a so-called \emph{algorithmic meta-classification}. This is a general statement saying that all problems that share some property or properties $P$ admit, over some classes of input restrictions $I$, a classification according to whether or not they have property $S$. If the input-restricted class has property $S$, then this problem is ``efficiently solvable''; otherwise it is ``computationally hard''. Throughout, we let these two notions depend on context; for example, efficiently solvable and computationally hard could mean being solvable in polynomial time and being \NP-complete, respectively. 

Algorithmic meta-classifications are less common than algorithmic meta-theorems, but 
let us mention two famous results. Grohe \cite{Gr07} proved that there is a polynomial-time algorithm for finite-domain constraint satisfaction problems whose left-hand input structure is restricted to $\mathcal{C}$ if and only if $\mathcal{C}$ has bounded treewidth (assuming $\mathrm{W}[1]\neq \mathrm{FPT}$). Bulatov~\cite{Bu17} and Zhuk~\cite{Zh20} proved that every finite-domain CSP$(H)$ is either polynomial-time solvable or \NP-complete, omitting any Ladner-like complexities in between.

Two well-known meta-classifications apply to the classes of ${\cal H}$-minor-free graphs and ${\cal H}$-topological-minor-free graphs. For a set ${\cal H}$ of graphs, these are the class of graphs $G$ where, starting from $G$, no graph $H \in {\cal H}$ can be obtained by a series of vertex deletions, edge deletions, and edge contractions, respectively a series of vertex deletions, edge deletions, and vertex dissolutions (see Section~\ref{s-prelims} for full definitions). 
Both are a consequence of a classic result of~\cite{RS86}; we refer to Appendix~\ref{a-planar} for proof details, but see also e.g.~\cite{KLMT11}.

\begin{theorem}\label{t-planar} 
Let $\Pi$ be a problem that is computationally hard on planar graphs, but efficiently solvable for every graph class of bounded treewidth. For any set of graphs~${\cal H}$, the problem~$\Pi$ on ${\cal H}$-minor-free graphs is efficiently solvable if ${\cal H}$ contains a planar graph (or equivalently, if the class of ${\cal H}$-minor-free graphs has bounded treewidth) and is computationally hard otherwise.
\end{theorem}

\begin{theorem}\label{t-topo} 
Let $\Pi$ be a problem that is computationally hard on planar subcubic graphs, but efficiently solvable for every graph class of bounded treewidth. For any set of graphs~${\cal H}$, the problem~$\Pi$ on ${\cal H}$-topological-minor-free graphs is efficiently solvable if ${\cal H}$ contains a planar subcubic graph (or equivalently, if the class of ${\cal H}$-topological-minor-free graphs has bounded treewidth) and is computationally hard otherwise.
\end{theorem}

Later in our paper, we will discuss many problems that satisfy the conditions of Theorems~\ref{t-planar} and~\ref{t-topo}. We refer, for example, to~\cite{FHL21,Mu17} for a number of problems that satisfy the conditions of Theorem~\ref{t-topo}, and thus also of Theorem~\ref{t-planar}, and that are \NP-complete even for planar subcubic graphs of high girth. 

On the other end of the spectrum lie ${\cal H}$-free graphs (or hereditary graphs). These are graphs $G$ where, starting from $G$, no graph $H \in {\cal H}$ can be obtained by a series of vertex deletions. These graph classes are much more complex in structure than ${\cal H}$-minor-free graphs and ${\cal H}$-topological-minor-free graphs and the family of hereditary graphs contains infinite descending chains of graph classes. This makes the task of finding algorithmic meta-theorems much harder. In fact, even algorithmic meta-theorems are unknown for the induced subgraph relation, even for a single forbidden graph~$H$. Indeed, even complexity dichotomies for $H$-free graphs are rare and only known for specific problems (see e.g.~\cite{BBJPPL21,GP14,Ka12,KP20}). Some attempts have been made to study complexity boundaries, e.g.~through the notion of boundary graph classes~\cite{Al03} (see also~\cite{ABKL07,KLMT11,Mu17}), but we are far from a broad understanding. Even a fundamental problem like {\sc Independent Set} is not known to be polynomial-time solvable on $H$-free graphs for many graphs~$H$ (see e.g.~\cite{GKPP22}).

Between ${\cal H}$-minor-free graphs and ${\cal H}$-topological-minor-free graphs on the one side and ${\cal H}$-free graphs on the other side, lies the class of ${\cal H}$-subgraph-free graphs. 
These are the graphs $G$ where, starting from $G$, no graph $H \in {\cal H}$ can be obtained by a series of vertex or edge deletions. 
In general, for any set ${\cal H}$ of graphs, 

\medskip
\noindent
 ${\cal H}$-minor-free graphs $\subseteq$ ${\cal H}$-topological-minor-free graphs $\subseteq$ ${\cal H}$-subgraph-free graphs $\subseteq$ ${\cal H}$-free graphs.

 \medskip
\noindent
Forbidden subgraphs represent many rich graph classes. 
For example, the classes of graphs of maximum degree at most~$r$ and $K_{1,r+1}$-subgraph-free graphs coincide; the class of graphs with girth larger than $g$ for some integer $g\geq 3$ coincides with the class of $(C_3,\ldots,C_g)$-subgraph-free graphs (and with the class of $(C_3,\ldots,C_g)$-free graphs);  
a class of graphs~${\cal G}$ has bounded treedepth if it is a subclass of $P_r$-subgraph-free graphs for some constant~$r$, and vice versa~\cite{NOM12}; 
and for every class~${\cal G}$ of degenerate or nowhere dense graphs~\cite{NO11}, there exists an integer $t$ such that every $G\in {\cal G}$ is $K_{t,t}$-subgraph-free (see~\cite{TV19} for a proof). 
Moreover, $H$-free graphs and $H$-subgraph-free graphs coincide if and only if $H = K_r$ for some integer $r \geq 1$.  This leads to a rich structural landscape.

Despite capturing natural graph classes, ${\cal H}$-subgraph-free graphs have been significantly less studied in the context of algorithms than the other classes. Yet, they exhibit extreme and unexpected jumps in problem complexity. For example, there exist problems that are PSPACE-complete in general but constant-time solvable for every ${\cal H}$-free graph class~\cite{MPS22} and thus for every ${\cal H}$-subgraph-free graph class, where ${\cal H}$ is any (possibly infinite) nonempty set of graphs. 
Another example is the {\sc Clique} problem, which is well-known to be \NP-hard (see~\cite{GJ79}). However, {\sc Clique} is polynomial-time solvable for all classes of ${\cal H}$-subgraph-free graphs. The reason is that the size of a largest clique is bounded by the number of vertices of a smallest graph in~${\cal H}$ and hence, one can apply brute force to find a largest clique in an ${\cal H}$-subgraph-free graph.

In contrast to ${\cal H}$-free graphs, however, some work has pointed to more complex dichotomy results being possible. Kami\'nski~\cite{Ka12} gave a complexity dichotomy for {\sc Max-Cut} restricted to ${\cal H}$-subgraph-free graphs, where ${\cal H}$ is any finite set of graphs.
Twenty years earlier, Alekseev and Korobitsyn~\cite{AK92} did the same for {\sc Independent Set}, {\sc Dominating Set} and {\sc Long Path}; see~\cite{GPR15} for a short, alternative proof (similar to the one of~\cite{Ka12} for {\sc Max-Cut}) for the classification for {\sc Independent Set} for $H$-subgraph-free graphs. In~\cite{GP14} the computational complexity of {\sc List Colouring} for ${\cal H}$-subgraph-free graphs has been determined for every finite set of graphs~${\cal H}$. 
More recently, Bodlaender et al.~\cite{BHKKOO20} determined the computational complexity of {\sc Subgraph Isomorphism} for $H$-subgraph-free graphs for all connected graphs~$H$ except the case where $H=P_5$, and they reduced all open ``disconnected'' cases to either $H = P_5$ or $H=2P_5$. 
However, even for a classical problem such as {\sc Colouring}, a complete complexity classification for $H$-subgraph-free graphs is far from settled~\cite{GPR15}. Many problems have not been studied in this context at all. 

Motivated by our apparent lack of understanding of ${\cal H}$-subgraph-free graphs, this paper embarks on a deeper investigation of the complexity of graph problems on these classes. 

\subsection{Our Results}
We pioneer an algorithmic meta-classification of ${\cal H}$-subgraph-free graphs. Before we define our framework, we first give some terminology. A class of graphs has bounded treewidth 
if there is a constant~$c$ such that every graph in it has treewidth at most~$c$. Recall that a graph is subcubic if every vertex has degree at most~$3$.
For an integer $k\geq 1$, the {\it $k$-subdivision} of an edge $e=uv$ of a graph replaces $e$ by a path of length $k+1$ with endpoints $u$ and $v$ (and $k$ new vertices). 
The {\it $k$-subdivision} of a graph~$G$ is the graph obtained from $G$ after $k$-subdividing each edge. 
For a graph class ${\cal G}$ and an integer~$k$, let ${\cal G}^k$ consist of the $k$-subdivisions of the graphs in ${\cal G}$.

A graph problem $\Pi$ is computationally hard
{\it under edge subdivision of subcubic graphs} if there is an integer~$k\geq 1$ such that:
if $\Pi$ is computationally hard for the class ${\cal G}$ of subcubic graphs, then $\Pi$ is computationally hard for ${\cal G}^{kp}$ for every $p\geq 1$; see Section~\ref{s-c3} for some reasons why we cannot simplify this definition.

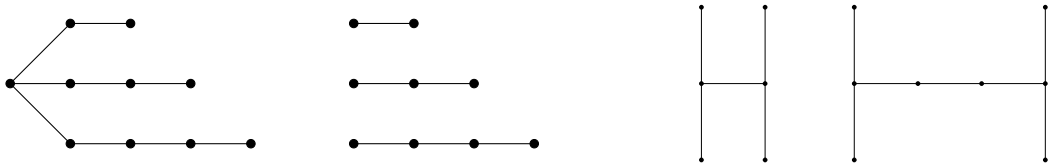
\begin{figure}[b]
\begin{minipage}[c]{0.3\textwidth}
\hspace*{1cm}
\begin{tikzpicture}[scale=0.8]
\draw (0,1)--(-1,1)--(-2,0)--(-1,-1)--(2,-1) (-2,0)--(1,0); \draw[fill=black] (-1,1) circle [radius=2pt] (0,1) circle [radius=2pt] (-2,0) circle [radius=2pt] (-1,0) circle [radius=2pt] (0,0) circle [radius=2pt] (1,0) circle [radius=2pt] (-1,-1) circle [radius=2pt] (0,-1) circle [radius=2pt] (1,-1) circle [radius=2pt] (2,-1) circle [radius=2pt];
\end{tikzpicture}
\end{minipage}
\qquad
\begin{minipage}[c]{0.2\textwidth}
\begin{tikzpicture}[scale=0.8] 
\draw (-1,1)--(0,1) 
(-1,-1)--(2,-1) (-1,0)--(1,0); \draw[fill=black] (-1,1) circle [radius=2pt] (0,1) circle [radius=2pt]  (-1,0) circle [radius=2pt] (0,0) circle [radius=2pt] (1,0) circle [radius=2pt] (-1,-1) circle [radius=2pt] (0,-1) circle [radius=2pt] (1,-1) circle [radius=2pt] (2,-1) circle [radius=2pt];
\end{tikzpicture}
\end{minipage}
\hspace{1.2cm}
\begin{minipage}[c]{0.4\textwidth}
\includegraphics[scale=0.3]{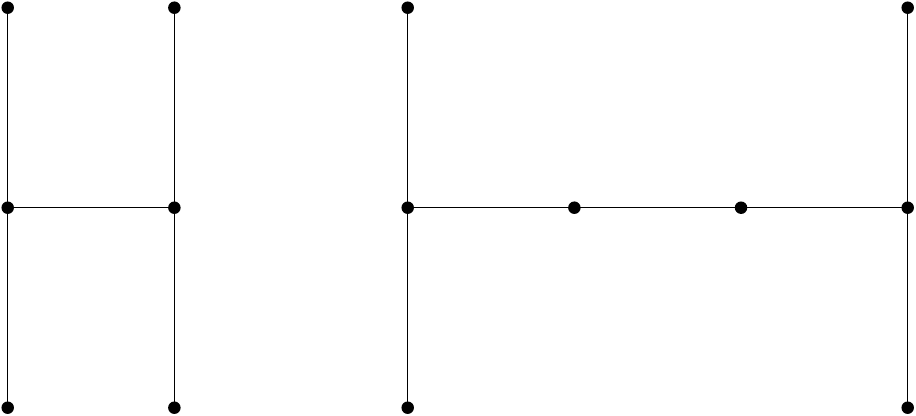}
\end{minipage}
\caption{Left: An example of a graph in ${\cal S}$ (the graph $S_{2,3,4}+P_2+P_3+P_4)$. Right: the graphs $\mathbb{H}_1$ and $\mathbb{H}_3$, where
$\mathbb{H}_1$ is the ``H''-graph, formed by an edge (the {\it middle edge}) joining the middle vertices of two $P_3$s, and $\mathbb{H}_i$ ($i\geq 2$) is obtained from $\mathbb{H}_1$ by $(i-1)$-subdividing the middle edge.} \label{f-st}
\end{figure}

Our framework contains every graph problem~$\Pi$  satisfying the following three conditions:

\begin{enumerate}
\item [{\bf C1.}] $\Pi$ is efficiently solvable for every graph class of bounded treewidth;
\item [{\bf C2.}] $\Pi$ is computationally hard 
 for the class of subcubic graphs; and
\item [{\bf C3.}] $\Pi$ is computationally hard under edge subdivision 
of subcubic graphs.
\end{enumerate}

\noindent
A problem $\Pi$ that satisfies conditions C1--C3 is called a {\it C123-problem}. 

For some $p,q,r\geq 1$, the {\it subdivided} claw $S_{p,q,r}$ is obtained from the {\it claw} (the $4$-vertex star $K_{1,3}$) after $(p-1)$-, $(q-1)$-, and $(r-1)$-subdividing its three edges respectively. 
The {\it disjoint union} $G_1+G_2$ of two vertex-disjoint graphs $G_1$ and $G_2$ is the graph $(V(G_1)\cup V(G_2),E(G_1)\cup E(G_2))$.  We now define the set ${\cal S}$, which plays an important role in our paper; see the left side of Fig.~\ref{f-st} for an example of a graph that belongs to ${\cal S}$.

\begin{definition}
The set ${\cal S}$ consists of all non-empty disjoint unions of a set of zero or more subdivided claws and paths.
\end{definition}

\noindent
Our main result is the following theorem that can be seen as the ``subgraph variant'' of Theorems~\ref{t-planar} and~\ref{t-topo}.
Note that it suggests, just like Theorems~\ref{t-planar} and~\ref{t-topo}, that boundedness of treewidth might be the underlying explanation for the polynomial-time solvability.

\begin{theorem}\label{t-dicho}
Let $\Pi$ be a C123-problem.
For any finite set of graphs ${\cal H}$, the problem $\Pi$ on ${\cal H}$-subgraph-free graphs is efficiently solvable if ${\cal H}$ contains a graph from ${\cal S}$ 
 (or equivalently, if the class of ${\cal H}$-subgraph-free graphs has bounded treewidth)
and computationally hard otherwise.
\end{theorem}

\noindent
As mentioned, the notions of efficiently solvable and computational hardness depend on context. 
The proof of the ``efficient'' part of Theorem~\ref{t-dicho} (in Section~\ref{s-dicho}) uses a well-known path-width result~\cite{BRST91}. To prove the ``hard'' part, we show that every problem satisfying C2 and C3 is hard for $(C_3,\ldots,C_i,K_{1,4},\mathbb{H}_1,\ldots,\mathbb{H}_i)$-subgraph-free graphs for every integer $i\geq 3$ (see the right side of Figure~\ref{f-st} for some examples of a subdivided ``H''-graph $\mathbb{H}_i$). As such, it is similar to the proofs for {\sc Max-Cut}~\cite{Ka12} and {\sc List Colouring}~\cite{GP14} and {\sc Independent Set} for $H$-subgraph-free graphs~\cite{GPR15} for finite~${\cal H}$) (the original proofs from~\cite{AK92} for {\sc Independent Set}, {\sc Dominating Set} and {\sc Long Path}, restricted to ${\cal H}$-free graphs for finite ${\cal H}$, are different and do not involve any direct path-width arguments).

\medskip
\noindent
{\bf Impact}
The impact of this framework is three-fold. First and foremost, we are able to provide a complete dichotomy for many problems on ${\cal H}$-subgraph-free graphs by showing they are C123-problems. In this way, we obtain a dichotomy between polynomial-time solvability and \NP-completeness for many well-known partitioning, covering and packing problems, network design problems and width parameter problems. We first show that computing the path-width and treewidth of a graph are C123-problems. We do the same for a number of covering and packing problems: {\sc (Independent) Odd Cycle Transversal}, {\sc $P_3$-Factor} and two variants of the {\sc Dominating Set} problem, namely {\sc Independent Dominating Set} and {\sc Edge Dominating Set}; the latter is polynomially equivalent to {\sc Minimum Maximal Matching}~\cite{Ha91}. We also show that {\sc Independent Set} (or equivalently, {\sc Vertex Cover}) and {\sc Dominating Set} are C123, and thus we recover the known classifications of~\cite{AK92}. We also show that {\sc List Colouring} is C123, and thus we re-obtain the classification of~\cite{GP14}. Next we prove (still in Section~\ref{s-npc}) that the following network design problems are all C123-problems: {\sc Edge/Node Steiner Tree}, {\sc (Induced) Disjoint Paths}, {\sc Long (Induced) Cycle}, {\sc Long (Induced) Path}, {\sc Max-Cut} and {\sc Edge/Node Multiway Cut}. Hence, we also recover the classification of~\cite{AK92} for {\sc Long Path} restricted to ${\cal H}$-subgraph-free graphs. We also include a reference to a subsequent result of~\cite{FLPR}, in which it was shown that {\sc Perfect Matching Cut} is C123.

We then consider the polynomial-time solvable problems {\sc Diameter} and {\sc Radius}, which are studied in fine-grained complexity. Here, Theorem~\ref{t-dicho} gives a distinction between almost-linear-time solvability versus not having a subquadratic-time algorithm under the Orthogonal Vectors Conjecture~\cite{Wi05} and Hitting Set Conjecture~\cite{AWW16}, respectively. The Orthogonal Vectors conjecture is implied by SETH~\cite{Wi05} and by the Hitting Set Conjecture~\cite{AWW16}; we refer to~\cite{Wi18} for more context on both conjectures.

These applications of Theorem~\ref{t-dicho}, as well as a number of applications of Theorem~\ref{t-planar} and~\ref{t-topo}, are summarized in Table~\ref{t-thetable}. A detailed comparison is deferred to Section~\ref{s-compare}.

\begin{table}[!tb]
\begin{center}
\hspace*{-5mm}
\setlength{\tabcolsep}{6pt}
\resizebox{!}{8cm}{
\begin{tabular}{llll:lllll}
\toprule
\textbf{graph problem} & \textbf{MF}  & \textbf{\mbox{TMF}} & \textbf{SF}\ \  & \textbf{plan.} & \textbf{\mbox{subc.~plan.}} & \textbf{C1} & \textbf{C2} & \textbf{C3}\\
\midrule
{\sc Path-Width} &yes  & yes & yes &  \cite{MS88} &  \cite{MS88} & \cite{BK96} & \cite{MS88} & T\ref{t-pwpwpw} \\ 
{\sc Tree-Width} &? &? &yes& ? & ? & \cite{Bo96} & \cite{BBJKLMOPS23} & T\ref{t-twtwtw} \\
{\sc Dominating Set}&yes  &yes &yes& \cite{GJ79} & \cite{GJ79} & \cite{AP89} & \cite{GJ79} & \cite{CC07} \\
{\sc Independent Dominating Set}&yes  &yes  &yes & \cite{CCJ90} & \cite{CCJ90} & \cite{TP97} & \cite{CCJ90} & \cite{CC07}\\
{\sc Edge Dominating Set}&yes  & yes & yes & \cite{YG80} & \cite{YG80} & \cite{ALS91} & \cite{YG80} & \cite{CC07}\\
{\sc Independent Set}&yes &yes & yes &  \cite{Mo01}  &  \cite{Mo01}  & \cite{AP89} & \cite{Mo01} & \cite{Po74}\\
{\sc Vertex Cover}&yes & yes &yes&  \cite{Mo01} &  \cite{Mo01} & \cite{AP89} & \cite{Mo01} & \cite{Po74}\\
{\sc Connected Vertex Cover}&yes  & no & no & \cite{GJ77} & no & \cite{ALS91} & no & triv \\
{\sc Feedback Vertex Set} &yes  & no & no &  \cite{Sp83} & no & \cite{ALS91} & no & triv \\
{\sc Independent Feedback Vertex Set} &yes  & no & no & \cite{Sp83} & no & \cite{TIZ15} & no & triv \\
{\sc Odd Cycle Transversal}&yes  & yes  &yes & T\ref{t-odd} & T\ref{t-odd} &\cite{ALS91} & T\ref{t-odd} & T\ref{t-odd}\\
{\sc Independent Odd Cycle Transv.}&yes &yes & yes& T\ref{t-odd}& T\ref{t-odd} & \cite{ALS91} & T\ref{t-odd} & T\ref{t-odd}\\
{\sc $C_5$-Colouring} &yes & yes & no& \cite{MS09} & T\ref{thm:C5-cubcubic-planar} & \cite{DP89} & \cite{GHN00} & no \\
{\sc $3$-Colouring} &yes  & no & no & \cite{Go08} & no & \cite{AP89} & no & triv \\
{\sc Star $3$-Colouring} &yes  & yes & no & \cite{ACKKR04} & \cite{BJMOPS20} & \cite{Co90} & \cite{BJMOPS20} & no \\
{\sc List Colouring}&yes  &yes  &yes & \cite{JS97} & T\ref{t-odd} & T\ref{t-odd} & T\ref{t-odd} & T\ref{t-odd} \\
{\sc $P_3$-Factor}&yes  &yes &yes &  \cite{XL21} &  \cite{XL21} & \cite{ABKL07} & \cite{ABKL07} & \cite{ABKL07}\\
{\sc Edge Steiner Tree}&yes &yes & yes & \cite{GJ77}    & T\ref{EST:C123} & \cite{ALS91}  & T\ref{EST:C123} & T\ref{EST:C123} \\
{\sc Node Steiner Tree}&yes &yes & yes & \cite{GJ77}   &  T\ref{EST:C123} & \cite{ALS91} & T\ref{EST:C123} & T\ref{EST:C123} \\
{\sc Steiner Forest}&no &no & no & \cite{GJ77}   &  T\ref{EST:C123} & no & T\ref{EST:C123} & T\ref{EST:C123} \\
{\sc Disjoint Paths}&yes  &yes & yes& \cite{MP93} & \cite{MP93} & \cite{Sc94} & \cite{MP93} & T\ref{t-dp-idp} \\			
{\sc Induced Disjoint Paths}&yes  &yes & yes & T\ref{t-dp-idp} & T\ref{t-dp-idp} &\cite{Sc94} & T\ref{t-dp-idp} & T\ref{t-dp-idp} \\
{\sc Long Cycle}&yes  &yes & yes& \cite{GJS74} & \cite{GJS74} & \cite{Bo93} & \cite{GJS74} & T\ref{t-lp-lip-lcl-lic}\\
{\sc Long Induced Cycle}&yes &yes &yes & \cite{GJS74} & \cite{GJS74} & \cite{JKT20} & T\ref{t-lp-lip-lcl-lic} & T\ref{t-lp-lip-lcl-lic}\\
{\sc Hamilton Cycle} &yes&yes & no& \cite{GJS74} & \cite{GJS74} & \cite{AP89} & \cite{GJT76} & no \\
{\sc Long Path}&yes   &yes& yes & \cite{GJS74} & \cite{GJS74} & \cite{Bo93} & \cite{GJS74} & T\ref{t-lp-lip-lcl-lic}\\
{\sc Long Induced Path}&yes&yes&yes & T\ref{t-lp-lip-lcl-lic} & T\ref{t-lp-lip-lcl-lic} &  \cite{JKT20} & T\ref{t-lp-lip-lcl-lic} & T\ref{t-lp-lip-lcl-lic}\\
{\sc Hamilton Path} &yes &yes  &no& \cite{GJS74} & \cite{GJS74} & \cite{AP89} & \cite{GJT76} & no \\
{\sc Max-Cut}& no&no&yes& no & no & \cite{ALS91} & \cite{Ya78} & \cite{Ka12} \\
{\sc Edge Multiway Cut}&yes &yes &yes & \cite{JMPPSvL22} & \cite{JMPPSvL22} & \cite{DLZ13} & \cite{JMPPSvL22} & T\ref{t-mwc} \\
{\sc Node Multiway Cut}&yes&yes&yes & \cite{JMPPSvL22} & \cite{JMPPSvL22} & \cite{ALS91} & \cite{JMPPSvL22} & T\ref{t-mwc} \\
{\sc Matching Cut} &yes &no &no & \cite{Bo09} & no & \cite{Bo09} & no & triv \\
{\sc Perfect Matching Cut} &yes &yes  &yes & \cite{BCD23} & \cite{BCD23} & \cite{LT21} & \cite{LT21} & \cite{FLPR} \\
{\sc Diameter} &no$^*$ &no$^*$  &yes& no & no & \cite{AWW16} & \cite{ED16} & L\ref{DiameterSubdivision}\\
{\sc Radius} &no$^*$ &no$^*$ &yes & no & no & \cite{AWW16} & \cite{ED16} & T\ref{t-d-r}\\
{\sc Subgraph Isomorphism} &no &no &no & \cite{BHKKOO20} & \cite{BHKKOO20} & no & \cite{BHKKOO20} & triv \\
{\sc Clique} &no&no&no& no & no & triv & no & triv\\
  \end{tabular}
}
\vspace*{3.5mm}
\caption{The minor framework (MF), topological minor framework (TMF) and subgraph framework (SF), with the conditions. If a problem satisfies the conditions of a meta-classification, we indicate this with ``yes''; if not, with a ``no''; and if unknown with a ``?''. 
A reference in a column is a references for a ``yes''-statement and ``triv'' means that the statement is trivial.
The problems are mainly chosen to illustrate the wide reach of the frameworks and their differences. 
A further discussion of the table, explaining the ``no''-statements, is in Section~\ref{s-compare}. 
$^*$For {\sc Diameter} and {\sc Radius}, the complexity on (subcubic) planar graphs is still open; however, a distinction almost-linear versus not subquadratic 
(as in Theorem~\ref{t-d-r})
 is not possible~\cite{GawrychowskiKMS21}.
}
\label{t-thetable}
\end{center}
\vspace*{-1cm}
\end{table}

The second impact of our framework is that we uncover several open questions in the literature. We subsequently resolve them in this work. An important and difficult open question turned out to be the complexity of {\sc Edge/Node Multiway Cut}, for which the classic results of Dahlhaus et al.~\cite{DJPSY94} shows \NP-completeness of the unweighted variant for planar graphs of maximum degree~$11$ (and claims an improved bound of~$6$). 
In a companion paper~\cite{JMPPSvL22}, we show \NP-completeness for {\sc Edge/Node Multiway Cut} on planar subcubic graphs, besting the earlier degree bound and showing these problems satisfy C2.
Moreover, we prove, as new results, that {\sc List Colouring}, {\sc Odd Cycle Transversal}, {\sc Independent Odd Cycle Transversal}, and {\sc $C_5$-Colouring} are \NP-complete for planar subcubic graphs, and thus satisfy C2. For the following problems we give explicit proofs to show that they satisfy C3: {\sc Path-Width}, {\sc Tree-Width}, {\sc Edge/Node Steiner Tree}, {\sc Edge/Node Multiway Cut} and {\sc Diameter}. Hence, our framework on ${\cal H}$-subgraph-free graphs shows the way towards new results.

The third impact of our framework is that it enables a structured investigation into complexity dichotomies for graph problems that do not satisfy some of the conditions, C1, C2 or~C3, particularly when only one is not satisfied. We call such problems C23, C13, or C12, respectively. This direction leads to many interesting new research questions. We are currently trying to determine new complexity classifications for a number of relevant problems in separate works, see e.g.~\cite{MPPSV,JMPPSV,BJMOPPSV}. This led to new insights into the complexity of well-studied problems such as {\sc Hamilton Cycle}~\cite{MPPSV}, {\sc Feedback Vertex Set}~\cite{JMPPSV}, and {\sc Steiner Forest}~\cite{BJMOPPSV}. For all these problems, the complexity classifications are different from the one in Theorem~\ref{t-dicho}. Hence, our framework has the potential to open a new and rich research area.

\subsection{Organization}
We start with some preliminaries in Section~\ref{s-prelims}. In Section~\ref{s-dicho} we prove Theorem~\ref{t-dicho}  
and consider several graph-theoretic consequences. In Section~\ref{s-npc} and~\ref{s-poly}, we then apply our framework to a wealth of problems. A discussion on limitations of the framework is provided in Section~\ref{s-limits}. We provide an extensive comparison of the applicability of the three frameworks (Theorem~\ref{t-planar}, \ref{t-topo}, and~\ref{t-dicho}) in Section~\ref{s-compare}. Finally, we conclude in Section~\ref{s-con} with a list of open problems and research directions. 

\section{Preliminaries} \label{s-prelims}

A graph $G$ contains a graph $H$ as a {\it subgraph} if $G$ can be modified to $H$ by a sequence of vertex deletions and edge deletions; if not, then $G$ is {\it $H$-subgraph-free}.
A graph $G$ contains $H$ as an {\it induced subgraph} if $G$ can be modified to $H$ by a sequence of only vertex deletions; if not, then $G$ is {\it $H$-free}. 
The {\it girth} of a graph $G$ that is not a forest is the length (number of edges) of a shortest cycle in $G$.

The {\it contraction} of an edge $e=uv$ in a graph replaces $u$ and $v$ by a new vertex that is made adjacent precisely to the former neighbours of $u$ and $v$ in $G$ (without creating multiple edges). If $v$ had degree~$2$ and its two neighbours in $G$ are non-adjacent, then we also say that we {\it dissolved} $v$ and call the operation the \emph{vertex dissolution} of $v$.
A graph $G$ contains $H$ as a {\it topological minor} (or as a {\it subdivision}) if $G$ can be modified to $H$ by a sequence of vertex deletions, vertex dissolutions and edge deletions; if not, then $G$ is {\it $H$-topological-minor-free}. A graph $G$ contains $H$ as a {\it minor} if $G$ can be modified to $H$ by a sequence of vertex deletions, edge deletions and edge contractions; if not, then $G$ is {\it $H$-minor-free}. 

For a set ${\cal H}$ of graphs, a graph $G$ is  {\it ${\cal H}$-subgraph-free} if $G$ is $H$-subgraph-free for every $H\in {\cal H}$. 
If ${\cal H}=\{H_1,\ldots,H_p\}$ for some integer $p\geq 1$, we also say that $G$ is {\it $(H_1,\ldots,H_p)$-subgraph-free}.
We also define the analogous notions of being ${\cal H}$-free, ${\cal H}$-topological-minor-free and ${\cal H}$-minor-free. A class of ${\cal H}$-free graphs is also said to be {\it hereditary}.

A \emph{tree decomposition} of a graph~$G=(V,E)$ is a pair~$(T,{\cal X})$ where $T$ is a tree and ${\cal X}$ is a collection of subsets of $V$ called {\it bags} such that the following holds. A vertex $i\in T$ is  a {\it node} and corresponds to exactly one bag $X_i\in {\cal X}$. The tree $T$ has the following two properties. First, for each~$v\in V$, the nodes of $T$ that contain~$v$ induce a non-empty connected subgraph of $T$. Second, for each edge~$vw\in E$, there is at least one node of~$T$ that contains both~$v$ and~$w$. The {\it width} of $(T,{\cal X})$ is one less than the size of the largest bag in ${\cal X}$. The {\it treewidth} of~$G$ is the minimum width of its tree decompositions. If we require $T$ to be a path, then we obtain the notions {\it path decomposition} and {\it path-width}.

A graph parameter $p$ {\it dominates} a parameter~$q$ if there is a function~$f$ such that $p(G)\leq f(q(G))$ for every graph~$G$. If $p$ dominates~$q$, but $q$ does not dominate $p$, then $p$ is {\it more powerful} than~$q$. If $p$ dominates~$q$ and vice versa, then we say that $p$ and $q$ are {\it equivalent}.
Note that every graph of path-width at most~$c$ has treewidth at most~$c$. However, the class of trees has treewidth~$1$, but unbounded path-width (see~\cite{Di05}). 
Hence, treewidth is more powerful than path-width.

\section{The Proof of Theorem~\ref{t-dicho}}\label{s-dicho}

We present a stronger result that will imply Theorem~\ref{t-dicho}. A graph class closed under edge deletion is also called {\it monotone}~\cite{ABKL07,BL02,KLMT11}. For a set of graphs~${\cal H}$,  the class of ${\cal H}$-subgraph-free graphs is {\it finitely defined} if ${\cal H}$ is a finite set. We say that a problem $\Pi$ is C1$'$D if $\Pi$ satisfies the following two conditions (see Fig.~\ref{f-st} for examples of the subdivided ``H''-graphs $\mathbb{H}_i$):

\begin{itemize}
\item [{\bf C1$'$}.] $\Pi$ is efficiently solvable for every finitely defined monotone graph class of bounded path-width;
\item [{\bf D}.] For every $i\geq 3$, $\Pi$ is computationally hard for the class of $(C_3,\ldots,C_i,K_{1,4}, \mathbb{H}_1,\ldots, \mathbb{H}_i)$-subgraph-free graphs.
\end{itemize}

\noindent
Our first theorem shows that the class of C1$'$D-problems is a proper superclass of the class of C123-problems.

\begin{theorem}\label{t-cd}
Every C123-problem is C1$'$D, but not every C1$'$D-problem is C123.
\end{theorem}

\begin{proof}
Let $\Pi$ be a C123-problem. Then $\Pi$ satisfies C1 and thus C1$'$. 
To show condition~D, let $i\geq 3$, and let ${\cal G}_i$ be the class of $(C_3,\ldots,C_i,K_{1,4}, \mathbb{H}_1,\ldots, \mathbb{H}_i)$-subgraph-free graphs. As $\Pi$ satisfies C2, $\Pi$ is computationally hard for the class~${\cal G}$ of subcubic graphs, that is, $K_{1,4}$-subgraph-free graphs. As $\Pi$ satisfies C3, there exists an integer~$k$, such that $\Pi$  is computationally hard for ${\cal G}^{kp}$ for every $p\geq 1$. We now choose $p$ to be sufficiently large, say $p=i+1$, such that ${\cal G}^{kp}$ is a subclass of ${\cal G}_i$. To show that the reverse statement does not necessarily hold, we define the following (artificial) example problem.

Let $\mathcal{B}$ be the set of all graphs obtained from a cycle after adding a new vertex made adjacent to precisely one vertex of the cycle. Then the problem {\sc $\mathcal{B}$-Modified List Colouring} takes as input a graph~$G$ with a list assignment~$L$ and asks whether $G$ simultaneously has a colouring respecting $L$ and has a connected component that is a graph from $\mathcal{B}$.

We now prove that {\sc $\mathcal{B}$-Modified List Colouring} is not C123 but is C1$'$D. We distinguish between ``being polynomial-time solvable'' and ``being \NP-complete''. 
We first observe that ${\cal B}$ satisfies the following four properties:

\begin{enumerate}
\item For every integer~$p$, the $p$-subdivision of any graph in $\mathcal{B}$ is not in $\mathcal{B}$.  
\item We can recognize whether a graph belongs to ${\cal B}$ in polynomial time.
\item Every graph in $\mathcal{B}$ admits a $3$-colouring.
\item For every finite set ${\cal H}$ disjoint from ${\cal S}$, there is an ${\cal H}$-subgraph-free graph in~${\cal B}$.
 \end{enumerate}

\noindent
Due to Property~1, {\sc $\mathcal{B}$-Modified List Colouring} does not satisfy C3. Hence, {\sc $\mathcal{B}$-Modified List Colouring} is not a C123-problem. We will prove that {\sc $\mathcal{B}$-Modified List Colouring} is C1$'$D. As {\sc List Colouring} is C123 by Theorem~\ref{t-odd}, it satisfies C1 and thus C1$'$. By Property~2, we can check in polynomial time if a graph has a connected component in ${\cal B}$. Hence, {\sc $\mathcal{B}$-Modified List Colouring} satisfies C1$'$. Below we prove that it also satisfies condition D.

Let $i\geq 3$, and let ${\cal G}_i$ be the class of $(C_3,\ldots,C_i,K_{1,4}, \mathbb{H}_1,\ldots, \mathbb{H}_i)$-subgraph-free graphs. 
As {\sc List Colouring} is C123, it follows from the first statement that it is also C1$'$D. Hence, {\sc List Colouring} is \NP-complete on ${\cal G}_i$.
Let $(G,L)$ be an instance of {\sc List Colouring} where $G$ is a graph from ${\cal G}_i$. We note that $\{C_3,\ldots,C_i,K_{1,4}, \mathbb{H}_1,\ldots, \mathbb{H}_i)\} \cap {\cal S}=\emptyset$.
Hence, by Property~4, there is an ${\mathcal H}$-subgraph-free graph~$B\in {\cal B}$. Let $G'=G+B$. Extend $L$ to a list assignment $L'$ by giving each vertex of $B$ list $\{1,2,3\}$. We claim that $(G,L)$ is a yes-instance of {\sc List Colouring} if and only if $(G',L')$ is a yes-instance of {\sc $\mathcal{B}$-Modified List Colouring}.

First suppose $G$ has a colouring respecting $L$. By Property~$3$, $B$ is $3$-colourable. As vertices of $B$ have list $\{1,2,3\}$, $G'$ has a colouring respecting $L'$. As $G$ has $B\in {\cal B}$ as a connected component, $(G',L')$ is a yes-instance of {\sc $\mathcal{B}$-Modified List Colouring}.
Now suppose that $(G',L')$ is a yes-instance of {\sc $\mathcal{B}$-Modified List Colouring}. Then, $G'$ has a colouring respecting $L'$, and thus $G$ has a colouring respecting $L$. We conclude that
{\sc $\mathcal{B}$-Modified List Colouring} satisfies D and is thus a C1$'$D-problem. As we already showed that {\sc $\mathcal{B}$-Modified List Colouring} is not C123, this proves the second statement of the theorem. \qed
\end{proof}

\noindent
We also need a theorem from Bienstock, Robertson, Seymour and Thomas.

\begin{theorem}[\cite{BRST91}]\label{t-bi}
For every forest $F$, all $F$-minor-free graphs have path-width at most $|V(F)|-2$.
\end{theorem}
 
 \noindent
 We now prove a result, which shows that the conditions C1$'$ and D are both necessary and sufficient. 
 
 \begin{theorem}\label{t-dicho4}
Let $\Pi$ be a problem. Then the following two statements are equivalent:
\begin{itemize}
\item [(i)] $\Pi$ is C1$'$D; and
\item [(ii)] for any finite set of graphs ${\cal H}$, the problem $\Pi$ on ${\cal H}$-subgraph-free graphs is efficiently solvable if ${\cal H}$ contains a graph from ${\cal S}$ 
and computationally hard otherwise.
\end{itemize}
\end{theorem}
 
 \begin{proof}
First assume that $\Pi$ is C1$'$D. 
Let ${\cal H}$ be a finite set of graphs. First suppose that ${\cal H}$ contains a graph~$H$ from ${\cal S}$.
Let $G$ be a ${\cal H}$-subgraph-free graph. As $G$ is ${\cal H}$-subgraph-free, $G$ is $H$-subgraph-free. It is known (see e.g.~\cite{GP14,GPR15}) that, for any graph $H'\in {\cal S}$, a $H'$-subgraph-free graph is also $H'$-minor-free. Hence, $G$ is $H$-minor-free. So by Theorem~\ref{t-bi}, $G$ has constant path-width at most $|V(H)|-2$, meaning we can solve $\Pi$ efficiently by C1$'$.

Now suppose that ${\cal H}$ contains no graph from ${\cal S}$. Let $H\in {\cal H}$. As $H\notin {\cal S}$, $H$ has a connected component containing a $K_{1,4}$ (or equivalently, a vertex of degree at least~$4$); or a cycle $C_h$ for some $h\geq 3$; or a graph $\mathbb{H}_i$ for some $i\geq 1$. Hence, the class of $H$-subgraph-free graphs contains the $K_{1,4}$-subgraph-free graphs; or $C_h$-subgraph-free graphs for some $h\geq 3$; or $\mathbb{H}_i$-subgraph-free graphs for some $i\geq 1$, each of which contains the $(C_3,\ldots,C_{j(H)},K_{1,4}, \mathbb{H}_1,\ldots, \mathbb{H}_{j(H)})$-subgraph-free graphs, where $j(H)=\max\{h,i\}$. 
Hence, the class of $H$-subgraph-free graphs contains the $(C_3,\ldots,C_{j(H)},K_{1,4}, \mathbb{H}_1,\ldots, \mathbb{H}_{j(H)})$-subgraph-free graphs. Consequently,
the class of ${\cal H}$-subgraph-free graphs contains the $(C_3,\ldots,C_{j^*},K_{1,4}, \mathbb{H}_1,\ldots, \mathbb{H}_{j^*})$-subgraph-free graphs, where $j^*=\max_{H\in {\cal H}}j(H)$ (note that $j$ exists, as ${\cal H}$ is finite).
As $\Pi$ satisfies D,  we find that $\Pi$ is computationally hard for $(C_3,\ldots,C_{j^*},K_{1,4}, \mathbb{H}_1,\ldots, \mathbb{H}_{j^*})$-subgraph-free graphs, and thus for ${\cal H}$-subgraph-free graphs.
 
\medskip
\noindent
Now assume that for any finite set of graphs ${\cal H}$, the problem $\Pi$ on ${\cal H}$-subgraph-free graphs is efficiently solvable if ${\cal H}$ contains a graph from ${\cal S}$  and computationally hard otherwise.
We first prove C1$'$. Let ${\cal H}$ be a finite set, such that the class of ${\cal H}$-subgraph-free graphs has bounded path-width. 
Recall that the latter holds if and only if ${\cal H}$ contains a graph from ${\cal S}$~\cite{RS84}. Hence, $\Pi$ satisfies C1$'$.

We now prove that condition $D$ holds. Let $i\geq 3$, and let ${\cal G}_i$ be the class of ${\cal H}$-subgraph-free graphs, where ${\cal H}=\{C_3,\ldots,C_i,K_{1,4}, \mathbb{H}_1,\ldots, \mathbb{H}_i\}$. Then ${\cal H}$ contains no graph from ${\cal S}$.
Hence, $\Pi$ is \NP-complete for ${\cal H}$-subgraph-free graphs. Hence, $\Pi$ satisfies~D. \qed
\end{proof}
 
 \noindent
 We are now ready to prove Theorem~\ref{t-dicho}, which we restate below.
 
\medskip
\noindent
{\bf  Theorem~\ref{t-dicho} (restated).}
{\it Let $\Pi$ be a C123-problem. For any finite set of graphs ${\cal H}$, the problem $\Pi$ on ${\cal H}$-subgraph-free graphs is efficiently solvable if ${\cal H}$ contains a graph from ${\cal S}$ 
 (or equivalently, if the class of ${\cal H}$-subgraph-free graphs has bounded treewidth)
and computationally hard otherwise.}
 
\begin{proof} 
The result follows from combining Theorems~\ref{t-cd} and~\ref{t-dicho4}, and the well-known fact
 that for a finite set of graphs ${\cal H}$, a class of ${\cal H}$-subgraph-free graphs has bounded path-width if and only if it has bounded treewidth if and only if ${\cal H}$ contains a graph from ${\cal S}$~\cite{RS84} (see e.g.~\cite{BL02,DP16}, for an explanation with respect to the more powerful parameter clique-width, and hence, replacing ``bounded pathwidth'' in C1$'$ by ``bounded treewidth'' or ``bounded clique-width'' yields the same equivalence as in Theorem~\ref{t-dicho4}). \qed
\end{proof} 

\medskip
\noindent
{\bf Remark.} We emphasize that we are not aware of any natural C1$'$D-problem that is not C123. As the conditions C1--C3 are more intuitive, we have therefore chosen to present our subgraph framework in terms of the C1--C3 conditions instead of the C1$'$-D conditions.

\section{Application to NP-Complete Problems}\label{s-npc}
We provide a complete dichotomy between polynomial-time solvability and \NP-completeness for many problems on ${\cal H}$-subgraph-free graphs by showing they are C123-problems. 
In Section~\ref{s-width}, we give examples of width parameter problems that are C123. In Section~\ref{s-cover} we give examples of partitioning, covering and packing problems that are C123.
In Section~\ref{s-network} we show the same for a number of network design problems. 
In fact we do a bit more. Namely, we also show that these problems belong to the minor and topological minor frameworks whenever the relevant \NP-completeness result applies to subcubic planar graphs. We will not explicitly remark this in the remainder of the section.

\subsection{Width Parameter Problems}\label{s-width}

Let {\sc Path-Width} and {\sc Tree-Width} be the problems of deciding for a given integer~$k$ and graph~$G$, if $G$ has path-width, or respectively, treewidth at most~$k$. 
We observe that it is unclear whether a path-width bound is maintained under subdivision, and thus proving property C3 for {\sc Path-Width} is non-trivial. We show a more specific result that is sufficient for our purposes. For {\sc Tree-width}, we can follow a more direct proof.

\begin{theorem}\label{t-pwpwpw}
{\sc Path-Width} is a C123-problem.
\end{theorem}

\begin{proof}
{\sc Path-Width} is linear-time solvable for every graph class of bounded treewidth~\cite{BK96} so satisfies C1.
It is \NP-complete for $2$-subdivisions of planar cubic graphs~\cite{MS88} so satisfies C2. It also satisfies C3, as we will prove the following claim:

\medskip
\noindent
{\it Claim.} A graph $G=(V,E)$ that is a $2$-subdivision of a graph $G''$ has path-width~$k$ if and only if the $1$-subdivision $G'$ of $G$ has path-width~$k$.

\medskip
\noindent
To prove the claim, we use the equivalence of path-width to the vertex separation number~\cite{Ki92}. 
We recall the definition. Let $L$ be a bijection between $V$ and $\{1,\ldots,|V|\}$, also called a \emph{layout} of $G$. Let $V_L(i) = \{u \mid L(u) \leq i\ \mbox{and}\ \exists v \in V: uv \in E\ \mbox{and}\ L(v) > i\}$. Then $\operatorname{vs}_L(G) = \max_{i \in \{1,\ldots,|V|\}} \{|V_L(i)|\}$ and $\operatorname{vs}(G) = \min_L \{\operatorname{vs}_L(G)\}$. 

Following Kinnersley~\cite{Ki92}, $G$ has a layout $L$ such that $\operatorname{vs}_L(G) = k$. In a $2$-subdivision, any edge $uv$ gets replaced by edges $ua$, $ab$, and $bv$, where $a$ and $b$ are new vertices specific to the edge $uv$. In a \emph{standard layout} $L$ for $G$, $L(a) = L(b) - 1$ and $L(u) < L(a)$. Following Ellis, Sudborough and Turner~\cite{EllisST94}, we may assume that $L$ is a standard layout. 

For some edge $uv$ of $G''$ and its $2$-subdivision into $ua$, $ab$, $bv$, consider a further subdivision of each of these three edges. Let $x$, $y$, $z$ be the newly created vertices respectively. Modify $L$ by placing $x$ directly before $a$, $y$ between $a$ and $b$, and $z$ directly after $b$. Let $L'$ denote the new layout. For simplicity and abusing notation, we use $L'(x) = L(a)-\frac{1}{2}$, $L'(y) = L(a)+\frac{1}{2} = L(b)-\frac{1}{2}$ and $L'(z) = L(b)+\frac{1}{2}$ to denote the positions of $x$, $y$ and $z$ in the new layout respectively. For any $i < L(a) - \frac{1}{2}$, $V_{L'}(i) = V_L(i)$, because $L(a) > i$ and $L'(x) > i$. Next, we observe that $V_{L'}(L'(x)) = V_{L'}(L(a) - \frac{1}{2}) = (V_L(L(a)) \setminus\{a\}) \cup \{x\}$, because $b$ follows after $a$ in $L$ and now $a$ follows after $x$ in $L'$. Hence, it has the same size as $V_L(L(a))$, at most $k$. Similarly, we can observe that $V_{L'}(L(a)) = V_L(L(a))$ (note that $L'(u) = L(u) < L(a)$), $V_{L'}(L(a) + \frac{1}{2}) = (V_L(L(a)) \setminus\{a\}) \cup \{y\}$, and $V_{L'}(L(b)) = (V_L(L(a)) \setminus\{a\}) \cup \{b\}$, which all have size at most $k$.
We then observe that $V_{L'}(L(b) + \frac{1}{2})$ is equal to $V_L(L(b))$ with $b$ replaced by $z$ if $b \in V_L(L(b))$. Similarly, for any $i > L(b)$, if $b \in V_{L}(i)$, we can replace it by $z$ to obtain $V_{L'}(i)$; otherwise, $V_{L'}(i) = V_L(i)$. Note that $a$ is never part of $V_L(i)$ for $i > L(b)$. In all cases, the size remains bounded by $k$. Hence, $\operatorname{vs}_{L'} \leq k$ and by the aforementioned equivalence between path-width and vertex separation number~\cite{Ki92}, $G'$ has path-width at most~$k$.

As subdivision cannot decrease path-width (or consider the converse, contraction cannot increase it), the claim, and thus C3, is proven.
This finishes the proof of Theorem~\ref{t-pwpwpw}. \qed
\end{proof}

\begin{theorem}
{\sc Tree-Width} is a C123-problem.
\label{t-twtwtw}
\end{theorem}

\begin{proof}
{\sc Tree-Width} is linear-time solvable for every graph class of bounded treewidth~\cite{Bo96}.
Very recently, it was announced that {\sc Tree-Width} is \NP-complete for cubic graphs~\cite{BBJKLMOPS23} so the problem satisfies C2.
It also satisfies C3, as we will prove the following claim:

\medskip
\noindent
{\it Claim.} A graph $G$ has treewidth~$k$ if and only if the $1$-subdivision of $G$ has treewidth~$k$.

\medskip
\noindent
Taking a minor of a graph does not increase its treewidth so the treewidth cannot decrease after subdividing an edge.
If a graph $G$ has treewidth~$1$, then $G$ remains a tree after subdividing an edge.
Suppose $G$ has treewidth at least~$2$. Let $(T,{\cal X})$ be a tree decomposition of $G$ with width~$k\geq 2$. Let $G'$ be the graph obtained by replacing an edge $uv$ with edges $ux$ and $xv$ for a new vertex~$x$.
Pick an arbitrary bag $B$ from ${\cal X}$ containing $u$ and $v$. Introduce the bag $\{u,v,x\}$ and make the corresponding node adjacent to the $B$-node. This yields a tree decomposition of $G'$ of width $k$, as $k\geq 2$ and the bag we added has size~$3$. Hence, the claim and thus the theorem is proven.\qed
\end{proof}

\subsection{Partitioning, Covering and Packing Problems}\label{s-cover}

The {\sc Vertex Cover} problem is to decide if a graph has a vertex cover of size at most~$k$ for some given integer~$k$.
The {\sc Independent Set} problem is to decide if a graph has an independent set of size at least~$k$ for some given integer~$k$. 
Note that {\sc Vertex Cover} and {\sc Independent Set} are polynomially equivalent. In the following theorem we recover the classification of~\cite{AK92}.

\begin{theorem}
{\sc Vertex Cover} and {\sc Independent Set} are C123-problems.
\end{theorem}

\begin{proof}
Both are linear-time solvable for graphs of bounded treewidth~\cite{AP89} so satisfy C1.
Both are \NP-complete for $2$-connected cubic planar graphs~\cite{Mo01} so satisfy C2. They also satisfy C3, as 
a graph $G$ on $m$ edges has an independent set of size~$k$ if and only if the $2$-subdivision of $G$ has an independent set of size $k+m$~\cite{Po74}. \qed
\end{proof}

\noindent
A set $D\subseteq V$ is {\it dominating} a graph $G=(V,E)$ if every vertex 
The {\sc (Independent) Dominating Set} problem is to decide if a graph has an (independent) dominating set of size at most~$k$ for some integer~$k$.
A set $F\subseteq E$ is an {\it edge dominating set} if every edge in $E\setminus F$ shares an end-vertex with an edge of $F$. The corresponding decision problem is {\sc Edge Dominating Set}. Recall that this problem is polynomially equivalent to {\sc Minimum Maximal Matching}~\cite{Ha91}.

The following theorem shows that both problems are C123, just like {\sc Dominating Set}; hence, we recover the classification of~\cite{AK92} for the latter problem.

\begin{theorem}
{\sc Dominating Set}, {\sc Independent Dominating Set} and {\sc Edge Dominating Set} are C123-problems.
\end{theorem}

\begin{proof}
{\sc Dominating Set}~\cite{AP89}, {\sc Independent Dominating Set}~\cite{TP97} and {\sc Edge Dominating Set}~\cite{ALS91} are linear-time solvable for graphs of bounded treewidth so satisfy C1.
{\sc Dominating Set}~\cite{GJ79}, {\sc Independent Dominating Set}~\cite{CCJ90} and {\sc Edge Dominating Set}~\cite{YG80} are \NP-complete for planar subcubic graphs so satisfy C2.
For showing C3 we use the following claim (see for example~\cite{CC07}) for a proof).
A graph $G$ with $m$ edges has a dominating set, independent dominating set or edge dominating set of size~$k$ if and only if the $3$-subdivision of $G$
has a dominating set, independent dominating set or edge dominating set, respectively, of size $k+m$. \qed
\end{proof}

\noindent
For a graph $G=(V,E)$, a function $c:V\to \{1,2\ldots\}$ is a {\it colouring} of $G$ if $c(u)\neq c(v)$ for every pair of adjacent vertices $u$ and $v$. If $c(V)=\{1,\ldots,k\}$ for some integer $k\geq 1$, then $c$ is also said to be a {\it $k$-colouring}. Note that a $k$-colouring of $G$ partitions $V$ into $k$ independent sets, which are called {\it colour classes}.
The {\sc $3$-Colouring} problem is to decide if a graph has a $3$-colouring.
A {\it list assignment} of a graph $G=(V,E)$ is a function $L$ that associates a {\it list of admissible colours} $L(u)\subseteq \{1,2,\ldots\}$ to each vertex $u\in V$. A colouring $c$ of $G$  {\it respects} ${L}$ if $c(u)\in L(u)$ for every $u\in V$. The {\sc List Colouring} problem is to decide if a graph $G$ with a list assignment~$L$ has a colouring that respects $L$. 
An {\it odd cycle transversal} in a graph $G=(V,E)$ is a subset~$S\subseteq V$ such that $G-S$ is bipartite. If $S$ is independent, then $S$ is an {\it independent odd cycle transversal}. The {\sc (Independent) Odd Cycle Transversal} problem is to decide if a graph has an (independent) odd cycle transversal of size at most~$k$ for a given integer~$k$. Note that a graph has an independent odd cycle transversal of size at most~$k$ if and only if it has $3$-colouring in which one of the colour classes has size at most~$k$.

Recall that $3$-{\sc Colouring} is not a C123-problem, as it does not satisfy C2 (see also Table~\ref{t-thetable}).
Indeed, because of Brooks' theorem~\cite{Br41}, it is polynomial-time solvable on subcubic graphs. 
This is in contrast to the situation for {\sc List Colouring}, {\sc Odd Cycle Transversal} and {\sc Independent Odd Cycle Transversal}: we show that all three problems are C123, and in this way recover the classification of~\cite{GP14} for {\sc List Colouring}.
 Our proof shows in particular that all three problems are in fact \NP-complete for planar subcubic graphs.

\begin{theorem}\label{t-odd}
{\sc List Colouring}, {\sc Odd Cycle Transversal} and {\sc Independent Odd Cycle Transversal} are C123-problems.
\end{theorem}

\begin{proof}
Jansen and Scheffler~\cite{JS97} proved that {\sc List Colouring} can be solved in linear time on graphs of bounded treewidth, so satisfies C1.
Both {\sc Odd Cycle Transversal} and {\sc Independent Odd Cycle Transversal} are linear-time solvable for graphs bounded treewidth~\cite{ALS91,FHRV08} so satisfy C1.

To prove C2 for all three problems, we modify the standard reduction to {\sc 3-Colouring} for planar graphs, which is from {\sc Planar $3$-Satisfiability} (we use the reduction form Proposition~2.27 of \cite{Go08}). This enables us to prove that all three problems are \NP-complete even for planar subcubic graphs.

The problem {\sc Planar $3$-Satisfiability} is known to be \NP-complete even when each literal appears in at most two clauses (see Theorem 2 in \cite{DDD16}). It is defined as follows. Given a CNF formula~$\phi$ that consists of a set $X= \{x_1,x_2,...,x_n\}$ of 
Boolean variables, and a set $C = \{C_1, C_2, . . . , C_m\}$ of two-literal or three-literal clauses over $X$, does there exist a truth assignment for $X$ such that each $C_j$ contains at least one true literal? If such a truth assignment exists, then $\phi$ is {\it satisfiable}.

Let $\phi$ be an instance of  {\sc Planar $3$-Satisfiability} on $n$ variables and $m$ clauses.
From $\phi$ we construct a graph $G$ as follows:

\begin{itemize}
    \item For $i=1,\ldots,n$, add the {\it literal vertices} ${x_i}$ and ${\overline{x_i}}$ and the edge $x_i\overline{x_i}$.
    \item Add a path $P$ of $2m$ vertices. The odd vertices represent {\it false} and the even vertices {\it true}.
    \item For each clause $C_j$, add a {\it clause gadget} as in Fig.~\ref{fig:3-col-gadget} with three labelled vertices $c_{j_1}, c_{j_2}, c_{j_3}$ as well as an {\it output} vertex labelled $c_j$.
    \item Fix an order of the literals $x_{j_1}, x_{j_2}, x_{j_3}$ of each three-literal clause $C_j$ and for $h=1,\ldots,3$, identify $x_{j_h}$ with $c_{j_h}$.
    \item Fix an order of the literals $x_{j_1}, x_{j_2}$ of each two-literal clause $C_j$ and for $h=1,\ldots,2$, identify $x_{j_h}$ with~$c_{j_h}$.
    \item Add an edge between $c_i$ and the $i$th odd vertex (representing false) of $P$.
    \item Add an edge between any unused input $c_{i_3}$ and the $i$th even vertex (representing true) on $P$.
\end{itemize}

\noindent
Note that $G$ is subcubic. Let us argue that $G$ is also planar. In {\sc Planar $3$-Satisfiability}, the bipartite incidence graph of clauses with variables is planar. We build $G$ so as to be planar in the following way. Uppermost we place the literals assigned to the inputs of the clause gadgets in just the manner prescribed in the bipartite incidence graph. Lowermost, we place the path of length $2m$ that will be joined on the odd vertices to the output nodes of the clause gadgets. 

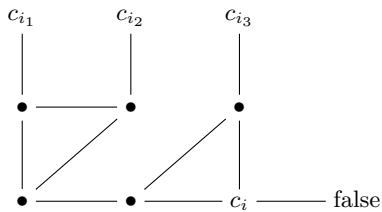
\begin{figure}[htbp]
\centering
\[
\xymatrix{
c_{i_1} \ar@{-}[d]& c_{i_2} \ar@{-}[d] & c_{i_3} \ar@{-}[d] & \\
\bullet \ar@{-}[r] \ar@{-}[d] & \bullet & \bullet \ar@{-}[d] & \\
\bullet \ar@{-}[r] \ar@{-}[ur]  & \bullet \ar@{-}[r] \ar@{-}[ur]  & c_i  \ar@{-}[r]  & \mathrm{false} \\
}
\]
\caption{The clause gadget in the reduction from {\sc (Planar) $3$-Satisfiability} to {\sc $3$-Colouring} drawn with an edge connecting the output node to a vertex representing false. The property that the gadget enforces is that not all of the three input nodes $c_{i_1}$, $c_{i_2}$, $c_{i_3}$ may be coloured the same as the vertex representing false.}
\label{fig:3-col-gadget}
\end{figure}

We will first prove that $G$ has an independent odd cycle transversal of size~$2m$ if and only if $\phi$ is satisfiable.
First suppose that $G$ has an independent odd cycle transversal~$S=\{v_1,\ldots,v_{2m}\}$ of size~$2m$. As $G$ contains $2m$ triangles, two in each clause gadget,  $T_1,\ldots,T_{2m}$, we may assume without loss of generality that $v_i\in V(T_i)$ for every $i\in \{1,\ldots,2m\}$.
Note that $G-S$ is bipartite by the definition of an odd cycle transversal.  Thus we can find a $3$-colouring of $G$ by colouring every vertex of $S$ with colour~$1$ and colouring every vertex of $G-S$ with colours~$2$ and~$3$. As each literal vertex belongs to $G-S$, it is assigned either colour $2$ or colour $3$, just like each vertex of $P$. Let us assume, without loss of generality, that the odd vertices on this path are coloured $3$. Hence, $2$ represents true and $3$ false. But now, by construction of the clause gadget, at least one of the vertices $c_{j_1}, c_{j_2}$ and $c_{j_3}$ is coloured $2$ and is identified with a literal for $j=1,\ldots m$, and therefore we deduce that $\phi$ is satisfiable.

Now suppose that $\phi$ is satisfiable. Colour the vertices of $P$ alternatingly with $3$ and $2$. In each clause, colour each true literal with colour~$2$ and each false or unused literal with colour $3$. Then, by construction of the clause gadget, we can extend this to a $3$-colouring of $G$. Let $S$ be the set of vertices of $G$ coloured~$1$. Then $S\subseteq V(T_1)\cup \cdots \cup V(T_{2m})$. Since we created a proper colouring, $S$ consists of exactly one vertex of each~$T_i$ and its vertices are pairwise non-adjacent. So $S$ is an independent odd cycle transversal of $G$ of size~$2m$. 

To prove C2 for {\sc List Colouring}, we use exactly the reduction above, with the literal vertices, any unused input $c_{i_3}$ and the vertices of $P$ assigned the list $\{2,3\}$, but all other vertices permitted to be any of the three colours.
Hence, as every list will be a subset of $\{1,2,3\}$, this result even holds for {\sc List $3$-Colouring}.

\medskip
\noindent
We now prove C2 for {\sc Independent Odd Cycle Transversal} and C3 for {\sc Odd Cycle Transversal} and {\sc Independent Odd Cycle Transversal}. To prove C3 for {\sc Odd Cycle Transversal}, we can just use the following claim (see for example~\cite{CHJMP18} for a proof):

\medskip
\noindent
{\it Claim.} The size of a minimum odd cycle transversal of $G$ is equal to the size of a minimum odd cycle transversal of the $2$-subdivision of $G$.

\medskip
\noindent
We now show C2 and C3 for {\sc Independent Odd Cycle Transversal} by proving that {\sc Independent Odd Cycle Transversal} is \NP-complete for $2p$-subdivisions of subcubic planar graphs.
Consider the subclass of planar subcubic graphs that correspond to instances of
 {\sc Planar $3$-Satisfiability} as defined in our proof for C2 for {\sc Odd Cycle Transversal}.
We now apply the above Claim sufficiently many times. In the graph~$G'$ resulting from the $2p$-subdivision, any minimum odd cycle transversal will also be an independent odd cycle transversal (by inspection of the proof for C2 for {\sc Odd Cycle Transversal}, because the cycles that were once triangles become further and further apart). 

By inspection of the proof of Lemma~3 in~\cite{GP14}, also {\sc List Colouring} satisfies C3.\qed
\end{proof}

\noindent
A {\it $P_3$-factor} or {\it perfect $P_3$-packing} of a graph $G=(V,E)$ with $|V|=3k$ for some integer $k\geq 1$ is a partition of~$V$ into subsets $V_1,\ldots,V_k$, such that each $G[V_i]$ is either isomorphic to $P_3$ or $K_3$. The corresponding decision problem, which asks whether a graph has such a partition, is known as {\sc $P_3$-Factor} or {\sc Perfect $P_3$-Packing}.
We show that {\sc $P_3$-Factor} is a C123-problem, a result which is essentially due to~\cite{ABKL07}.

\begin{theorem}
{\sc $P_3$-Factor} is a C123-problem.
\end{theorem}

\begin{proof}
This follows from combining Proposition~1 of~\cite{ABKL07} for showing C1 with Lemma~12 of~\cite{ABKL07} for showing C2 and C3 (with $k=3$). Recently, Xi and Lin~\cite{XL21} proved that {\sc $P_3$-Factor} is \NP-complete even for claw-free planar cubic graphs, which also proves C2. \qed
\end{proof}

\subsection{Network Design Problems}\label{s-network}

\noindent
A {\it (vertex) cut} of a graph $G=(V,E)$ is a partition $(S,V\setminus S)$ of $V$. The {\it size} of~$(S,V\setminus S)$ is the number of edges with one end in $S$ and the other in $V\setminus S$. 
The {\sc Max-Cut} problem is to decide if a graph has a cut of size at least~$k$ for some integer~$k$. By combining the next result with Theorem~\ref{t-dicho}, we recover the classification of~\cite{Ka12}.

\begin{theorem}[\cite{Ka12}]
{\sc Max-Cut} is a C123-problem.
\end{theorem}

\begin{proof}
{\sc Max-Cut} is linear-time solvable for graphs of bounded treewidth~\cite{ALS91} and \NP-complete for subcubic graphs~\cite{Ya78} so satisfies C1 and C2.
A cut $C$ of a graph $G$ is {\it maximum} if $G$ has no cut of greater size. Kami\'nski~\cite{Ka12} proved that a graph $G=(V,E)$ has a maximum cut of size at least~$c$ if and only if the $2$-subdivision of $G$ has a maximum cut of size at least~$c+2|E|$. This shows C3. \qed
\end{proof}

\noindent
Let $G=(V,E)$ be a graph.
A set $M\subseteq E$ is a {\it perfect matching} if no two edges in $M$ share an end-vertex and moreover, every vertex of the graph is incident to an edge of $M$.
A set $M\subseteq E$ is an {\it edge cut} of $G$ if it is possible to partition $V$ into two sets $B$ and $R$, such that $M$ consists of all the edges with one end-vertex in $B$ and the other one in $R$. 
A set $M\subseteq E$ is a {\it perfect matching cut} of $G$ if $M$ is a perfect matching that is also an edge cut. The {\sc Perfect Matching Cut} is to decide if a graph has a perfect matching cut. 
Lucke et al.~\cite{FLPR} recently showed that {\sc Perfect Matching Cut} is C123. 

\begin{theorem}[\cite{FLPR}]\label{t-mc}
{\sc Perfect Matching Cut} is a C123-problem.
\end{theorem}

\begin{proof}
 Le and Telle~\cite{LT21} observed that {\sc Perfect Matching Cut} is polynomial-time solvable for graphs of bounded treewidth. In the same paper~\cite{LT21}, they also proved that for every integer~$g\geq 3$, it is \NP-complete even for subcubic bipartite graphs of girth at least $g$. Hence, {\sc Perfect Matching Cut} satisfies C1 and C2. The \NP-completeness proof in~\cite{LT21} implicitly showed that to get C3 we may take $k=4$ (see also~\cite{FLPR}).
 
We note that C2 also follows for {\sc Perfect Matching Cut} from a recent result of Bonnet, Chakraborty and Duron~\cite{BCD23}, who proved that {\sc Perfect Matching Cut} is \NP-complete even for $3$-connected subcubic planar graphs.
\qed
\end{proof}

\noindent
Given a graph $G$ and 
a set of terminals $T\subseteq V(G)$, and an integer $k$, the problems {\sc Edge (Node) Steiner Tree} are to decide if $G$ has a subtree containing all the terminals of $T$, such that the subtree has at most~$k$ edges (vertices). 
We give explicit proofs that {\sc Node Steiner Tree} and {\sc Edge Steiner Tree} are \NP-complete on planar subcubic graphs and that this is maintained under subdivision, leading to these two problems being C123-problems. 

\begin{theorem}\label{EST:C123}
{\sc Edge} and {\sc Node Steiner Tree}  are C123-problems.	
\end{theorem}

\begin{proof}  
As the two variants are equivalent (on unweighted graphs), we only consider {\sc Edge Steiner Tree}, which is linear-time solvable for graphs of bounded treewidth~\cite{ALS91} so satisfies C1.
For showing C2, we reduce from {\sc Edge Steiner Tree}, 
which is \NP-complete even for grid graphs~\cite{GJ77}, and thus for planar graphs.
 
Let $(G,T,k)$ be an instance, where $G$ is a planar graph with $|V(G)|=n$.
We build a planar subcubic graph $G'$ where we replace each node $v$ in $G$ with a rooted binary tree $T_v$ in which there are $n$ leaf vertices (so the tree contains at most $2n$ nodes and is of depth $\lceil \log n \rceil$).
  	For each edge $e=uv$ of $G$, add to $G'$ a path $e'$ of length $4n^2$ between some a leaf of $T_u$ and a leaf of $T_v$ (ensuring that each leaf is incident with at most one such path). If $v$ in $G$ is in $T$, then the root vertex of $T_v$ is a terminal in~$G'$ (and these are the only terminals in $G'$ and form the set $T'$).
We note that $G$ is planar subcubic, and we claim that $(G,T,k)$ is a yes-instance if and only if $(G',T',4n^2 \cdot k+2n^2)$ is a yes-instance.
  	
First suppose $G$ has a Steiner tree $S$ with at most $k$ edges. We build a Steiner tree $S'$ in $G'$: if $e=uv$ is in $S$, then we add to $S'$ a path that comprises $e'$ and paths that join the roots of~$T_u$ and $T_v$ to $e'$.  
The sum of the lengths of these paths, additional to the $4n^2 \cdot k$, is bounded above by $2 \cdot n \cdot \log n \leq 2n^2$.
  	
Now suppose $G'$ has a Steiner tree $S'$ with at most $4n^2 \cdot k+2n^2$ edges. 
 We build a tree $S$ in~$G$: if $e=uv$ and $e'$ is in $S'$, we add $e$ to $S$.  
Then~$S$ is a Steiner tree in~$G$.  As the length of a path from $T_u$ to $T_v$ is $4n^2$, the sum of the lengths of all such paths in $S'$ is a whole multiple of $4n^2$, so $|E(S)|\leq k$.
 
 To prove C3, it suffices to show the following claim:
 
 \medskip
 \noindent
 {\it Claim.} A graph $G$ has an edge Steiner tree for terminals $T$ of size at most $k$ if and only if the $1$-subdivision of $G$ has an edge Steiner tree for terminals $T$ of size at most $2k$.
 
 \medskip
 \noindent
 In order to see this, let $G'$ be the $1$-subdivision of $G$. 
  Let $e_1$ and $e_2$ be the two edges obtained from subdividing an edge $e\in E(G)$.
  Given a Steiner tree $S$ of $G$ with at most $k$ edges, we obtain a Steiner tree of $G'$ with at most $2k$ edges by replacing each edge $e$ of $S$ with $e_1$ and~$e_2$.  Given a Steiner tree $S'$ of $G'$ with at most $2k$ edges, we may assume that for any edge $e$ of $G$, either neither or both of $e_1$ and $e_2$ are in $S'$; if $S'$ contains only one it can safely be discarded.
To obtain a Steiner tree of $G$ with at most $k$ edges, include each edge $e$ if both $e_1$ and $e_2$ are in $S'$. \qed
  \end{proof}
  
 \noindent
 In the {\sc Edge Multiway Cut} problem, also known as {\sc Multiterminal Cut}, we are given an input graph $G=(V,E)$, a subset $T$ of its vertices, and an integer $k$. The goal is to decide whether there exists a set $S\subseteq E$ such that $|S| \leq k$ and for any pair of vertices $\{u,v\}\in T$, $G \setminus S$ does not contain a path between $u$ and $v$. In the {\sc Node Multiway Cut} problem, 
we ask for a set $S \subseteq V \setminus T$ such that $|S| \leq k$ and for any pair of vertices $\{u,v\}\in T$, $G \setminus S$ does not contain a path between $u$ and $v$.

\begin{theorem}
{\sc Edge} and {\sc Node Multiway Cut} are C123-problems.
\label{t-mwc}
\end{theorem}

\begin{proof}
{\sc Edge Multiway Cut} is linear-time solvable for graphs of bounded treewidth~\cite{DLZ13} (also following~\cite{ALS91}) and \NP-complete for planar subcubic graphs~\cite{JMPPSvL22} so satisfies C1 and C2. It satisfies C3 as well, as we will prove the following claim:

\medskip
\noindent
{\it Claim.} A graph $G$ has an edge multiway cut for a set of terminals $T$ of size at most $k$ if and only if the $1$-subdivision of $G$ has an edge multiway cut for $T$ of size at most $k$.

\medskip
\noindent
In order to see this, let $G'$ be the $1$-subdivision of $G$.
For each edge $e$ in $G$, there exist two edges in $G'$. If an edge of $G$ is in an edge multiway cut for $G$ and $T$, then it suffices to pick only one of the two edges created from it in $G'$ to disconnect the paths $e$ lies on. Vice versa, if an edge $e'$ of $G'$ is in an edge multiway cut for $G'$ and $T$, then it suffices to pick the unique corresponding edge in $G$ to disconnect the paths $e'$ lies on. 

\medskip
\noindent
We now turn to {\sc Node Multiway Cut}, which is linear-time solvable for graph classes of bounded treewidth~\cite{ALS91} (it is an extended monadic second-order linear extremum problem)
and \NP-complete for planar subcubic graphs~\cite{JMPPSvL22} so satisfies C1 and C2. It satisfies C3, as we will prove the following claim:

\medskip
\noindent
{\it Claim.} A graph $G$ has a node multiway cut for a set of terminals $T$ of size at most $k$ if and only if its $1$-subdivision has a node multiway cut for $T$ of size at most $k$. 

\medskip
\noindent
In order to see this, let $G'$ be the $1$-subdivision of $G$. We observe that subdividing any edge of a graph does not create new connections between terminals. Moreover, we can assume that none of the newly introduced vertices of the subdivision are used in some optimal solution for $G'$ and $T$. \qed
\end{proof}
  
\noindent 
Given a graph $G$ and disjoint vertex pairs $(s_1, t_1), (s_2,t_2), \ldots (s_k, t_k)$, the {\sc Disjoint Paths} problem is to decide if $G$ has $k$ pairwise vertex-disjoint paths from $s_i$ to $t_i$ for every $i$.
We obtain the {\sc Induced Disjoint Paths} problem if the paths are required to be mutually induced;
a set of paths $P^1,\ldots, P^k$
is {\it mutually induced}  if $P^1,\ldots, P^k$ are pairwise vertex-disjoint and there is no edge between a vertex of some $P^i$ and a vertex of some $P^j$ if $i\neq j$.

\begin{theorem}
{\sc Disjoint Paths} and {\sc Induced Disjoint Paths} are C123-problems.
\label{t-dp-idp}
\end{theorem}

\begin{proof}
The {\sc Disjoint Paths} problem is linear-time solvable for graphs of bounded treewidth~\cite{Sc94} and \NP-complete for 
planar subcubic graphs~\cite{MP93} so satisfies C1 and C2.
The {\sc Induced Disjoint Paths} problem is solvable in polynomial time for graphs of bounded mim-width~\cite{JKT20} and thus for bounded treewidth~\cite{Vat12}, so it satisfies C1.
Let $G'$ be the $1$-subdivision of a subcubic graph $G$ and let ${\cal T}$ be a set of disjoint vertex pairs.
Then, $(G,{\cal T})$ is a yes-instance of {\sc Disjoint Paths} if and only if
 $(G',{\cal T})$ is a yes-instance of {\sc Disjoint Paths} if and only if $(G',{\cal T})$ is a yes-instance of {\sc Induced Disjoint Paths}. Hence, C2 is satisfied for {\sc Induced Disjoint Paths} as well and C3 is satisfied for both problems. \qed
\end{proof}

\noindent
The {\sc Long Path} and {\sc Long Induced Path} are to decide for a given graph $G$ and integer~$k$, whether $G$ contains $P_k$ as a subgraph or induced subgraph, respectively.
The {\sc Long Cycle} and {\sc Long Induced Cycle} problems are defined similarly.
By combining the next result with Theorem~\ref{t-dicho}, we recover the classification of~\cite{Ka12} for {\sc Long Path}. The classification of {\sc Long Cycle} was not made explicit in~\cite{ABKL07}, but is implicitly there (combine Proposition~1 of~\cite{ABKL07} with Lemma~12 of~\cite{ABKL07}).

\begin{theorem}
{\sc Long Path}, {\sc Long Induced Path}, {\sc Long Cycle} and {\sc Long Induced Cycle} are C123-problems.
\label{t-lp-lip-lcl-lic}
\end{theorem}

\begin{proof}
Bodlaender~\cite{Bo93} proved that {\sc Long Path} and {\sc Long Cycle} are polynomial-time solvable for graphs of bounded treewidth. Hence, {\sc Long Path} and {\sc Long Cycle} satisfy C1. As {\sc Hamilton Path} (so {\sc Long Path} with $k=|V(G)|$) and {\sc Hamilton Cycle} (so {\sc Long Cycle} with $k=|V(G)|$) are \NP-complete for subcubic planar graphs~\cite{GJS74}, {\sc Long Path} and {\sc Long Cycle} satisfy C2.

Let $G'$ be the $1$-subdivision of a subcubic graph $G$.
Now the following holds: $(G,k)$ is a yes-instance of {\sc Long Path} if and only if  $(G',2k)$ is a yes-instance of {\sc Long Path} if and only if  $(G',2k)$ is a yes-instance of {\sc Long Induced Path}.
Hence, C2 is satisfied for {\sc Long Induced Path} as well, and C3 is satisfied for both problems. Moreover, {\sc Long Induced Path} satisfies C1; it is even polynomial-time solvable for graphs of bounded mim-width~\cite{JKT20}. We can make the same observations for {\sc Long Cycle} and {\sc Long Induced Cycle}. \qed
\end{proof}

\section{Application to Polynomial-Time Solvable Problems}\label{s-poly}
We give two examples of polynomial-time solvable problems where Theorem~\ref{t-dicho} gives a distinction between almost-linear-time solvability versus not having a subquadratic-time algorithm (conditional under appropriate hardness hypotheses). 
Let $d(u,v)$ denote the distance between $u$ and $v$ in a graph~$G$. The {\it eccentricity} of $u\in V$ is $e(u) = \max_{v \in V} d(u, v)$. The {\it diameter} of~$G$ is the maximum eccentricity and the {\it radius} the minimum eccentricity. The {\sc Diameter} and {\sc Radius} problems are to find the diameter and radius, respectively, of a graph.
We need a lemma.
	
\begin{lemma}\label{DiameterSubdivision}
Let $G'$ be the $2$-subdivision of a graph $G$ with diameter~$d$.
Let $d'$ be the diameter of $G'$. Then $3d\leq d' \leq 3d+2$.
\end{lemma}

\begin{proof}
Under edge-subdivision, the shortest path between two original vertices does not change, it is only of longer length. As the path between two adjacent vertices in $G$ gets length~$3$ in $G'$, use any diametral pair in $G$ to find that $d'\geq 3d$.

Let $u$ and $v$ be two vertices of $V'$. If $u$ and $v$ belong to $G$, then they are of distance at most~$3d$ in $G'$.
If one of them, say $u$, belongs to $V$ and the other one, $v$, belongs to $V'\setminus V$, then they are of distance at most $3d+1$ in $G'$, as any vertex in $V'\setminus V$ is one step away from some vertex in $V$ and the diameter is $d$ in~$G$. If $u$ and $v$ both belong to $V'\setminus V$, then $u$ is adjacent to some vertex $w_u\in V$ and $v$ is adjacent to some vertex $w_v\in V$.
As the diameter is $d$ in $G$, vertices $w_u$ and $w_v$ lie at distance at most $3d$ from each other in $G'$. Hence, in this case, $d(u,v)\leq 3d+2$ in~$G'$. To summarize, the diameter of $G'$ is at most $3d + 2$. \qed
\end{proof}

\begin{theorem}
Both {\sc Diameter} and {\sc Radius} are C123-problems.
\label{t-d-r}
\end{theorem}

\begin{proof}
Both are solvable in $n^{1+o(1)}$ time for graphs of bounded treewidth~\cite{AWW16} and thus satisfy Condition C1.
Both also satisfy C2. Evald and Dahlgaard~\cite{ED16} proved that for subcubic graphs, no subquadratic algorithm exists for {\sc Diameter} exists under the Orthogonal Vectors Conjecture~\cite{Wi05}, and no subquadratic algorithm exists for {\sc Radius} exists under the Hitting Set Conjecture~\cite{AWW16}. 
From the construction in the proof of the latter result, we observe that any constant subdivision of all edges of the graph does not affect the correctness of the reduction, i.e., the parameter~$p$ in the construction can be increased appropriately to account for the subdivisions of the other edges. Hence, {\sc Radius} satisfies C3.
By Lemma~\ref{DiameterSubdivision}, {\sc Diameter} satisfies C3 as well. \qed
\end{proof}

\section{Limitations of our Framework} \label{s-limits}
We give two limitations of our framework.

\subsection{Forbidding An Infinite Number of Subgraphs}\label{s-formany}
We observe that in Theorems~\ref{t-planar} and~\ref{t-topo}, the set of graphs~${\cal H}$ is allowed to have infinite size.  However, the set of graphs~${\cal H}$ in Theorems~\ref{t-dicho} and~\ref{t-dicho4} cannot be allowed to have infinite size. 
This is because there exist infinite sets ${\cal H}$ such that

\begin{itemize}
\item [1.] ${\cal H}$ contains no graphs from $\cal{S}$.
\item  [2.] All C123-problems are efficiently solvable on $\cal{H}$-subgraph-free graphs.	
\end{itemize}

\noindent
To illustrate this, we give two examples. See, e.g.~\cite{Ka12}, for another example.

\medskip
\noindent
{\it Example 1.} Let ${\cal H}$ be the set of cycles ${\cal C}$. No graph from ${\cal C}$ belongs to ${\cal S}$.
Every ${\cal C}$-subgraph-free graph is a forest and thus has treewidth~$1$. Hence, every C123-problem is efficiently solvable on the class of ${\cal C}$-subgraph-free graphs (as it satisfies condition C1). 

\medskip
\noindent
{\it Example 2.} Let ${\cal H}=\{\mathbb{H}_1,\mathbb{H}_2,\ldots \}$; see also Fig.~\ref{f-st}.
No graph from ${\cal H}$ belongs to ${\cal S}$.
Every ${\cal H}$-subgraph-free graph~$G$ is $\mathbb{H}_1$-minor-free. By Theorem~\ref{t-bi}, $G$ has path-width, and thus treewidth, at most~$4$. Hence, every C123-problem is efficiently solvable on the class of ${\cal H}$-subgraph-free graphs. 

\subsection{Relaxing Condition C3}\label{s-c3}

In C3 we require the class ${\cal G}$ to be subcubic. In this way we are able to show in Theorem~\ref{t-cd} that every C123-problem~$\Pi$ satisfies condition~D, that is, for every $i\geq 3$, $\Pi$ is computationally hard for  
the class of $(C_3,\ldots,C_{\ell},K_{1,4},\mathbb{H}_1,\ldots,\mathbb{H}_\ell)$-subgraph-free graphs.\footnote{The aforementioned papers~\cite{FHL21,Mu17} show a number of problems to be \NP-complete for planar subcubic graphs of high girth, whereas we consider subcubic graphs of high girth that, instead of being planar, do not contain any small subdivided ``H''-graph as a subgraph.}
If we allow ${\cal G}$ to be any graph class instead of requiring ${\cal G}$ to be subcubic, then we can no longer show this, and hence the proof of Theorem~\ref{t-dicho} no longer holds in that case. That is, following the same arguments we can only construct a graph class that due to C2, is either $K_{1,4}$-subgraph-free (or equivalently, subcubic) or, due to C3, is $(C_3,\ldots,C_{\ell},\mathbb{H}_1,\ldots,\mathbb{H}_\ell)$-subgraph-free. Consequently, in that case, we can only obtain the dichotomy for ${\cal H}$-subgraph-free graphs if $|{\cal H}|=1$. 
This relaxation could potentially lead to a classification of more problems. However, so far, we have not identified any problems that belong to this relaxation but not to our original framework.

We also note that the existence of a problem-specific integer~$k\geq 1$ in C3 is needed. For instance, a $1$-division is bipartite and some computationally hard problems, such as {\sc Independent Set}, become efficiently solvable on bipartite graphs. In the proofs in Section~\ref{s-npc}, $k$ takes on values $1$, $2$, $3$ and $4$.

\section{Comparison between the Three Frameworks} \label{s-compare}
In this section, we provide an extensive discussion and comparison of the three frameworks in this paper: Theorem~\ref{t-planar}, \ref{t-topo}, and~\ref{t-dicho}. See also Table~\ref{t-thetable}.

\medskip
\noindent
{\bf Belonging to All Three Frameworks.}
Apart from {\sc Max-Cut} and possibly {\sc Tree-Width}, all C123-problems from Section~\ref{s-npc} are \NP-complete for planar subcubic graphs, and thus also satisfy the conditions of Theorems~\ref{t-planar} and~\ref{t-topo}.
In the proofs of Section~\ref{s-npc} we made explicit observations about this. The complexity of {\sc Tree-Width} is still open for planar graphs and planar subcubic graphs. 
It is also still open whether {\sc Diameter} and {\sc Radius} allow a distinction between almost-linear-time solvability versus not having a subquadratic-time algorithm on planar and subcubic planar graphs.

\medskip
\noindent
{\bf Not Falling Under Any of the Three Frameworks.}
Every problem that is \NP-complete for graphs of bounded treewidth does not satisfy any of the frameworks. An example is the aforementioned {\sc Subgraph Isomorphism} problem, which is \NP-complete even for input pairs $(G_1,G_2)$ that are linear forests (see, for example,~\cite{BHKKOO20} for a proof) and thus have tree-width~$1$. 
As another example, the {\sc Steiner Forest} problem is to decide for a given integer~$k$, graph $G$ and set of pairs of terminal vertices $S=\{(s_1,t_1),\ldots,(s_p,t_p)\}$, if $G$ has a subforest $F$ with at most $k$ edges,  such that $s_i$ and $t_i$, for every $i\in \{1,\ldots,p\}$, belong to the same connected component of $F$. It is readily seen that {\sc Steiner Forest} generalizes {\sc Edge Steiner Tree}: take all pairs of vertices of $T$ as terminal pairs to obtain an equivalent instance of {\sc Steiner Forest}.
Hence, {\sc Steiner Forest} is \NP-complete on planar subcubic graphs and this is maintained under subdivision, due to Theorem~\ref{EST:C123}. As {\sc Steiner Forest} is \NP-complete on graphs of treewidth~$3$~\cite{BHM11}, {\sc Steiner Forest} does not belong to any of the three frameworks. We refer to~\cite{BJMOPPSV} for a partial complexity classification of {\sc Steiner Forest} on $H$-subgraph-free graphs.

As an example on the other extreme end,  the {\sc Clique} problem does not fall under the subgraph framework either. This is because {\sc Clique} is polynomial-time solvable for ${\cal H}$-subgraph-free graphs for every set of graphs ${\cal H}$. We also note that {\sc Clique} is polynomial-time solvable for planar graphs (as every clique in a planar graph has size at most~$4$). Hence, it does not belong under the minor and topological minor frameworks either.

\medskip
\noindent
{\bf Only Belonging to the Minor Framework.}
We observe that every problem that satisfies the conditions of Theorem~\ref{t-topo} also satisfies the conditions of Theorem~\ref{t-planar}. However, there exist problems that satisfy the conditions of Theorem~\ref{t-planar} but not those of Theorems~\ref{t-topo} and~\ref{t-dicho}. For example, {\sc $3$-Colouring} satisfies C1 (this even holds for {\sc List Colouring}~\cite{JS97}), and it is \NP-complete even for $4$-regular planar graphs~\cite{Da80}. However, {\sc $3$-Colouring} does not satisfy the conditions of Theorems~\ref{t-topo} and~\ref{t-dicho}, as {\sc $3$-Colouring} is 
polynomial-time solvable for subcubic graphs due to Brooks' Theorem~\cite{Br41}. 

We can also take the problems {\sc Connected Vertex Cover}, {\sc Feedback Vertex Set} and {\sc Independent Feedback Vertex Set}. It is known that all three problems satisfy C1~\cite{ALS91}. Moreover,  {\sc Connected Vertex Cover}~\cite{GJ77} and {\sc Feedback Vertex Set}~\cite{Sp83} are \NP-complete for 
planar graphs of maximum degree at most~$4$. By taking $1$-subdivisions, the same holds for {\sc Independent Feedback Vertex Set}.
However, unlike the related problems {\sc Vertex Cover} and {\sc Odd Cycle Transversal}, the three problems do not satisfy the conditions of Theorems~\ref{t-topo} and~\ref{t-dicho}. This is because {\sc Connected Vertex Cover}~\cite{UKG88}, {\sc Feedback Vertex Set}~\cite{UKG88} and {\sc Independent Feedback Vertex Set}~\cite{JMPPSV} are polynomial-time solvable for subcubic graphs. Munaro~\cite{Mu17} showed that even {\sc Weighted Feedback Vertex Set} is polynomial-time solvable for subcubic graphs. Finally, the same holds for {\sc Matching Cut}, which satisfies C1~\cite{Bo09} and is \NP-complete for planar graphs of girth~$5$~\cite{Bo09} but polynomial-time solvable for subcubic graphs~\cite{Ch84}.

\medskip
\noindent
{\bf Only Belonging to the Minor and Topological Minor Frameworks.}
We also know of problems that satisfy the conditions of Theorem~\ref{t-topo} (and thus of Theorem~\ref{t-planar}) but not those of Theorem~\ref{t-dicho}. For example, {\sc Hamilton Cycle} is solvable in polynomial-time for graphs of bounded treewidth~\cite{AP89}, so satisfies C1, and it is \NP-complete for planar subcubic graphs~\cite{GJT76} (even if they are also bipartite and have arbitrarily large girth~\cite{Mu17}). Hence, {\sc Hamilton Cycle} satisfies the conditions of Theorem~\ref{t-topo}, and also satisfies C2. However, unlike its generalization {\sc Long Cycle}, which is C123, {\sc Hamilton Cycle} does not satisfy C3~\cite{MPPSV}, so it is not a C123-problem. The same holds for {\sc Hamilton Path} (which contrasts the C123-property of {\sc Long Path}).

To give another example, {\sc Star $3$-Colouring} is to decide if a graph~$G$ has a $3$-colouring such that the union of every two colour classes  induces a {\it star forest} (forest in which each connected component is a star). 
This problem is known to be \NP-complete even for subcubic planar subgraphs of arbitrarily large fixed girth~\cite{BJMOPS20}, but does not satsify C3~\cite{MPPSV}, so is not C123.

To give a final example of a problem that satisfies the conditions of Theorems~\ref{t-planar} and~\ref{t-topo} but not those of Theorem~\ref{t-dicho}, we can consider the $C_5$-{\sc Colouring} problem. This problem is to decide if a given graph allows a homomorphism to $C_5$. It is known to be \NP-complete on both subcubic graphs~\cite{GHN00} and planar graphs~\cite{MS09}. 
In order to show \NP-completeness for subcubic planar graphs, one can take the gadget of MacGillivray and Siggers~\cite{MS09} and augment it with a degree reduction gadget. As explained in Appendix~\ref{a-c5}, where we give a full proof, a suitable gadget appears in the arXiv version of~\cite{CHRSZ19}.\footnote{The use of this gadget for this purpose was proposed to us by Mark Siggers.} However, $C_5$-{\sc Colouring} does not satisfy C3~\cite{MPPSV}, so it not C123.

\medskip
\noindent
{\bf Only Belonging to the Subgraph Framework.}
Finally, there exist problems that satisfy the conditions of Theorem~\ref{t-dicho}, and thus are C123, but that do not satisfy the conditions of Theorems~\ref{t-planar} and~\ref{t-topo}. Namely,
{\sc Max-Cut} is polynomial-time solvable for planar graphs~\cite{Ha75} (and thus also for planar subcubic graphs). However, we show in Section~\ref{s-npc} that {\sc Max-Cut} satisfies the conditions of Theorem~\ref{t-dicho}, that is, is a C123-problem.

\section{Conclusions}\label{s-con}

By giving a meta-classification, we were able to unify a number of known results from the literature and give new complexity classifications for a variety of graph problems on classes of graphs characterized by a finite set~${\cal H}$ of forbidden subgraphs. Similar frameworks existed (even for infinite sets~${\cal H}$) already for the minor and topological minor relations, whereas for the subgraph relation, only some classifications for specific problems existed~\cite{AK92,GP14,Ka12}. 
 We showed that many problems belong to all three frameworks, and also that there exist problems that belong to one framework but not to (some of) the others. 

In order to have stronger hardness results for our subgraph framework, we considered the unweighted versions of these problems. However, we note that most of the vertex-weighted and edge-weighted variants of these problems satisfy C1 as well; see~\cite{ALS91}. We finish this section by setting out some directions for future work.

\subsection{Refining and Extending the Subgraph Framework}

We describe three approaches for refining or extending the subgraph framework. First, in the proof of  Theorem~\ref{t-cd} we gave an example of a C1'D-problem, namely {\sc ${\cal B}$-Modified List Colouring}, that is not C123.
However, this example is rather artificial. To increase our understanding of the conditions C1--C3 of our framework,  addressing the following question would be helpful.

\begin{open}\label{o-natural}
Do there exist any natural graph C1'D-problems that are not C123-problems?
\end{open}

\noindent
As a second approach, we recall from Section~\ref{s-c3} that we cannot relax condition~C3 by allowing the class ${\cal G}$ to be an arbitrary graph class instead of being subcubic. If we do this nevertheless, we are only able to obtain a dichotomy for ${\cal H}$-subgraph-free graphs if $|{\cal H}|=1$. This relaxation could potentially lead to a classification of more problems and we pose the following open problem.

\begin{open}
Can we classify more problems for $H$-subgraph-free graphs by no longer demanding that the class~${\cal G}$ in C3 is subcubic?
\end{open}

\noindent
So far, we have not identified any problems that belong to the relaxation but not to our original framework.

Recall that the set of forbidden graphs ${\cal H}$ is allowed to have infinite size in Theorems~\ref{t-planar} and~\ref{t-topo}.
For any infinite set of graphs ${\cal H}$, a C123-problem on ${\cal H}$-subgraph-free graphs is still efficiently solvable if ${\cal H}$ contains a graph~$H$ from ${\cal S}$. However, a C123-problem may no longer be computationally hard for ${\cal H}$-subgraph-free graphs if ${\cal H}$ has infinite size, as shown in Section~\ref{s-formany} with some examples. Hence, as a third approach for extending the subgraph framework, we propose the following problem. This problem was also posed by Kami\'nski~\cite{Ka12}, namely for the C123-problem {\sc Max-Cut}.

\begin{open}\label{o-infinite}
Can we obtain dichotomies for C123-problems restricted to ${\cal H}$-subgraph-free graphs when~${\cal H}$ is allowed to have infinite size?
\end{open}

\noindent
In order to solve Open Problem~\ref{o-infinite}, we need a better understanding of the treewidth of ${\cal H}$-subgraph-free graphs when ${\cal H}$ has infinite size. In recent years, such a study has been initiated for the induced subgraph relation; see, for example,~\cite{AACHS,ACV22,Ku,We19} for many involved results in this direction.

\subsection{Finding More Problems Falling under the Three Frameworks}

\medskip
\noindent
There still exist many natural problems for which it is unknown whether they belong to the minor, topological minor or subgraph framework. For the first two frameworks, we recall the following open problems, which have been frequently stated as open problems before.

\begin{open}
Determine the computational complexity of {\sc Tree-Width} for planar graphs and for planar subcubic graphs.
\end{open}

\begin{open}
Determine the fine-grained complexity of {\sc Diameter} and {\sc Radius} for planar graphs and for planar subcubic graphs.
\end{open}

\noindent
We now turn to the subgraph framework. We showed that {\sc Tree-Width} and {\sc Path-Width} are C123, but further investigation might reveal more such problems that fit the subgraph framework.

\begin{open}
Do there exist other width parameters with the property that the problem of computing them is C123?
\end{open}

\noindent
We also made a detailed comparison between the minor, topological minor and subgraph frameworks (see Section~\ref{s-compare}). To increase our general understanding of the complexity of graph problems, it would be interesting to find more problems that either belong to all frameworks or just to one or two. In particular, we are not aware of any problem that belongs to the minor an subgraph frameworks, but not to the topological minor framework. Such a problem (if it exists) must be computationally hard for planar graphs and subcubic graphs, but efficiently solvable for subcubic planar graphs.

\subsection{Dropping One of the Conditions C1, C2, or C3}

Another highly interesting direction is to investigate if we can obtain new complexity dichotomies for computationally hard graph problems that do not satisfy one of the conditions, C1, C2 or~C3. We call such problems C23, C13, or C12, respectively. 

Some interesting progress has recently been made on such problems (see e.g.~\cite{BJMOPPSV,JMPPSV,MPPSV}).
However, we note that in general, obtaining complete classifications is challenging for C12-, C13- and C23-problems. In particular, we need a better understanding of the structure of $P_r$-subgraph-free graphs and $\mathbb{H}_i$-subgraph-free graphs (recall that $\mathbb{H}_i$ is a subdivided ``H''-graph). Recall that a graph is $P_r$-subgraph-free if and only if it is $P_r$-(topological)-minor-free. Hence, if a problem is open for the case where $H=P_r$ for one of the frameworks, then it is open for all three of them.

To illustrate the challenges with an example from the literature, consider the aforementioned {\sc Subgraph Isomorphism} problem. This problem takes as input two graphs $G_1$ and $G_2$. Hence, it does not immediately fit in our framework, but one could view it as a C23-problem. The question is whether $G_1$ is a subgraph of~$G_2$. Recall that the {\sc Subgraph Isomorphism} problem is \NP-complete even for input pairs $(G_1,G_2)$ that are linear forests and thus even have path-width~$1$. Yet, even a classification for $H$-subgraph-free graphs was not straightforward; recall that Bodlaender et al.~\cite{BHKKOO20} settled the computational complexity of {\sc Subgraph Isomorphism} for $H$-subgraph-free graphs except if  $H=P_5$ or $H=2P_5$. These cases are open for the minor and topological minor frameworks as well due to the above observation (which also holds for linear forests).

\subsection{The Induced Subgraph Relation}\label{s-induced}

We finish our paper with some remarks on the induced subgraph relation. As mentioned, there exist ongoing and extensive studies on boundary graph classes (cf.~\cite{Al03,ABKL07,KLMT11,Mu17}) and treewidth classifications (cf.~\cite{AACHS,ACV22,Ku,We19}) in the literature. We note that for the induced subgraph relation, it is also useful to check C2 and C3. Namely, let $\Pi$ be a problem satisfying C2 and C3. For any finite set of graphs ${\cal H}$, the problem $\Pi$ on ${\cal H}$-free graphs is computationally hard if ${\cal H}$ contains no graph from ${\cal S}$. This follows from the same arguments as in the proof of Theorem~\ref{t-dicho4}.\footnote{The reason is that for any integer~$k$ and a sufficiently large integer~$\ell$, the class of subcubic $(C_3,\ldots,C_\ell,\mathbb{H}_1,\ldots,\mathbb{H}_k)$-free graphs coincides with the class of subcubic $(C_3,\ldots,C_\ell,\mathbb{H}_1,\ldots,\mathbb{H}_k)$-subgraph-free graphs.} Hence, if we aim to classify the computational complexity of problems satisfying C2 and C3 for $H$-free graphs (which include all C123-problems), then we may assume that $H\in {\cal S}$. For many of such problems, such as {\sc Independent Set}, this already leads to challenging open cases.

As mentioned, we currently do not know even any algorithmic meta-theorem for the induced subgraph relation, not even for a single forbidden graph~$H$.  However, a recent result of Lozin and Razgon~\cite{LR22} provides at least an initial starting point. To explain their result, the {\it line graph} of a graph $G$ has vertex set~$E(G)$ and an edge between two vertices $e_1$ and $e_2$ if and only if  $e_1$ and~$e_2$ share an end-vertex in $G$. Let ${\cal T}$ be the class of line graphs of graphs of ${\cal S}$. Lozin and Razgon~\cite{LR22} showed that for any finite set of graphs~${\cal H}$, the class of ${\cal H}$-free graphs has bounded treewidth if and only if ${\cal H}$ contains a complete graph, a complete bipartite graph, a graph from~${\cal S}$ and a graph from ${\cal T}$.
Their characterization leads to the following theorem, which could be viewed as a first meta-classification for the induced subgraph relation.

\begin{theorem}\label{t-induced} 
Let $\Pi$ be a problem that is \NP-complete on every graph class of unbounded treewidth, but polynomial-time solvable for every graph class of bounded treewidth.
For every finite set of graphs~${\cal H}$, the problem~$\Pi$ on ${\cal H}$-free graphs is polynomial-time solvable if ${\cal H}$ contains a  complete graph, a complete bipartite graph, a graph from~${\cal S}$ and a graph from ${\cal T}$, and it is \NP-complete otherwise.
\end{theorem}

\noindent
Note that by the aforementioned result of Hickingbotham~\cite{Hi22}, we may replace ``treewidth'' by ``path-width'' in Theorem~\ref{t-induced}.
However, currently, we know of only one problem that satisfies the conditions of Theorem~\ref{t-induced}, namely {\sc Weighted Edge Steiner Tree}~\cite{BBJPPL21},  where we allow the edges to have weights. As we showed, even  {\sc Edge Steiner Tree} (the unweighted version) is a C123-problem. Even though the conditions of Theorem~\ref{t-induced} are very restrictive, we believe the following open problem is still interesting.

\begin{open}
Determine other graph problems that satisfy the conditions of Theorem~\ref{t-induced}.
\end{open}

\smallskip
\noindent
{\it Acknowledgments.}
The authors wish to thank Hans Bodlaender, Daniel Lokshtanov, Vadim Lozin and Mark Siggers for helpful suggestions.


\appendix

\section{The Proof of Theorems~\ref{t-planar} and~\ref{t-topo}}\label{a-planar}

Both Theorems~\ref{t-planar} and~\ref{t-topo} follow immediately from the following classical result of Robertson and Seymour.

\begin{theorem}[\cite{RS86}]\label{t-rs86}
For every planar graph $H$, all $H$-minor-free graphs have tree-width at most $c_H$ for some constant~$c_H$ that only depends on the size of $H$.
\end{theorem}

\noindent
Here is the (known) proof of Theorem~\ref{t-planar}.

\medskip
\noindent
{\bf Theorem~\ref{t-planar} (restated).}
{\it  Let $\Pi$ be a problem that is computationally hard on planar graphs, but efficiently solvable for every graph class of bounded treewidth. For any set of graphs~${\cal H}$, the problem~$\Pi$ on ${\cal H}$-minor-free graphs is efficiently solvable if ${\cal H}$ contains a planar graph (or equivalently, if the class of ${\cal H}$-minor-free graphs has bounded treewidth) and is computationally hard otherwise.}

\begin{proof}
For some integer~$p\geq 1$, let ${\cal H}$ be a set of graphs, where we allow ${\cal H}$ to have infinite size. 
First assume that ${\cal H}$ contains a planar graph $H_1$. 
As $G$ is ${\cal H}$-minor-free, $G$ is $H_1$-minor-free.
We now apply Theorem~\ref{t-rs86} to find that
$G$ has treewidth bounded by some integer $c_{H_1}$, which is a constant as $H_1$ is a fixed graph. We conclude that the class of ${\cal H}$-minor-free graphs has bounded treewidth. Hence, by our assumption on $\Pi$, we can solve $\Pi$ efficiently for the class of ${\cal H}$-minor-free graphs.

Now assume that ${\cal H}$ contains no planar graph. As planar graphs are closed under taking vertex deletions, edge deletions and edge contractions, they are closed under taking minors.
This means that the class of planar graphs is a subclass of the class of ${\cal H}$-minor-free graphs. Hence, by our assumption on $\Pi$, we find that $\Pi$ is computationally hard for the class of ${\cal H}$-minor-free graphs.

It is well known that for a set of graphs ${\cal H}$, a class of ${\cal H}$-minor-free graphs has bounded treewidth if and only if ${\cal H}$ contains a planar graph.
These facts follow directly from results of Robertson and Seymour~\cite{RS84} (see e.g.~\cite{BL02,DP16}, where this is explained with respect to the more general parameter clique-width).
\qed
\end{proof}

\noindent
Here is the (known) proof of Theorem~\ref{t-topo}.

\medskip
\noindent
{\bf Theorem~\ref{t-topo} (restated).}
{\it Let $\Pi$ be a problem that is computationally hard on planar subcubic graphs, but efficiently solvable for every graph class of bounded treewidth. For any set of graphs~${\cal H}$, the problem~$\Pi$ on ${\cal H}$-topological-minor-free graphs is efficiently solvable if ${\cal H}$ contains a planar subcubic graph (or equivalently, if the class of ${\cal H}$-topological-minor-free graphs has bounded treewidth) and is computationally hard otherwise.}

\begin{proof}
For some integer~$p\geq 1$, let ${\cal H}$ be a set of graphs, where we allow ${\cal H}$ to have infinite size. 
First assume that ${\cal H}$ contains a planar subcubic graph $H_1$. 
As $G$ is ${\cal H}$-topological-minor-free, $G$ is $H_1$-topological-minor-free. 
As $H_1$ is subcubic, this means that $G$ is even $H_1$-minor-free.
We now apply Theorem~\ref{t-rs86} to find that
$G$ has treewidth bounded by some integer $c_{H_1}$, which is a constant as $H_1$ is a fixed graph. We conclude that the class of ${\cal H}$-subgraph-free graphs has bounded treewidth. Hence, by our assumption on $\Pi$, we can solve $\Pi$ efficiently for the class of ${\cal H}$-topological-minor-free graphs.

Now assume that ${\cal H}$ contains no planar subcubic graph. As planar subcubic graphs are closed under taking vertex deletions, vertex dissolutions and edge deletions, they are closed under taking topological minors.
This means that the class of planar subcubic graphs is a subclass of the class of ${\cal H}$-topological-minor-free graphs. Hence, by our assumption on $\Pi$, we find that $\Pi$ is computationally hard for the class of ${\cal H}$-topological-minor-free graphs.

It is well known that for a set of graphs ${\cal H}$, a class of ${\cal H}$-topological-minor-free graphs has bounded treewidth if and only if ${\cal H}$ contains a planar subcubic graph.
These facts follow directly from results of Robertson and Seymour~\cite{RS84} (see e.g.~\cite{BL02,DP16}, where this is explained with respect to the more general parameter clique-width).
\qed
\end{proof}

\section{Hardness of $\mathbf{C_5}$-Colouring for Subcubic Planar Graphs}\label{a-c5}

\begin{figure}[!]
\centering
\includegraphics[width=0.25\textwidth]{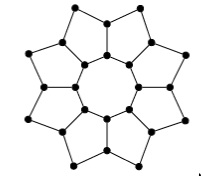}
\caption{Degree reduction gadget from Theorem~\ref{thm:C5-cubcubic-planar} with $d=8$.}
\label{fig:C5-cubcubic-planar}
\end{figure} 

We show the following result.

\begin{theorem}
{\sc $C_5$-Colouring} is \NP-complete on subcubic planar graphs.
\label{thm:C5-cubcubic-planar}
\end{theorem}
\begin{proof}
It is known from \cite{MS09} that {\sc $C_5$-Colouring} is \NP-complete on planar graphs. Let us introduce a degree reduction gadget communicated to us by Mark Siggers (a similar one appears in the proof of Theorem 4.3 from the arXiv version of~\cite{CHRSZ19}). The gadget in question is a collection of (some even number) $d$ copies of $C_5$, joined to one another in sequence by a single overlapping edge, such that the last is joined to the first to form a cycle. The resulting object appears like a flower and is drawn in Figure~\ref{fig:C5-cubcubic-planar} for the case $d=8$. In any homomorphism from this gadget to $C_5$, the outermost $d$ vertices must all be mapped to the same vertex of $C_5$. Thus, we may replace each vertex of degree $d>3$ with such a gadget, using once each these outermost vertices of the gadget as a surrogate for the original vertex. \qed
\end{proof}

\end{document}